\documentclass[14pt,a4paper]{article}
\usepackage{graphicx}

\usepackage{amsfonts}
\usepackage{amssymb}
\usepackage{amsmath}
\usepackage{amsthm}
\usepackage{amscd}

\def\vbar{\mathchoice{\vrule height6.3ptdepth-.5ptwidth.8pt\kern- .8pt}
{\vrule height6.3ptdepth-.5ptwidth.8pt\kern-.8pt} {\vrule
height4.1ptdepth-.35ptwidth.6pt\kern-.6pt} {\vrule
height3.1ptdepth-.25ptwidth.5pt\kern-.5pt}}
\def\fudge{\mathchoice{}{}{\mkern.5mu}{\mkern.8mu}}
\def\bbc#1#2{{\rm \mkern#2mu\vbar\mkern-#2mu#1}}
\def\bbb#1{{\rm I\mkern-3.5mu #1}}
\def\bba#1#2{{\rm #1\mkern-#2mu\fudge #1}}
\def\bb#1{{\count4=`#1 \advance\count4by-64 \ifcase\count4\or\bba
A{11.5}\or \bbb B\or\bbc C{5}\or\bbb D\or\bbb E\or\bbb F \or\bbc
G{5}\or\bbb H\or \bbb I\or\bbc J{3}\or\bbb K\or\bbb L \or\bbb
M\or\bbb N\or\bbc O{5} \or \bbb P\or\bbc C{5}\or\bbb B\or\bbc
S{4.2}\or\bba T{10.5}\or\bbc U{5}\or \bba V{12}\or\bba
W{16.5}\or\bba X{11}\or\bba Y{11.7}\or\bba Z{7.5}\fi}}

\newtheorem{thm}{Theorem}[section]
\newtheorem{cor}[thm]{Corollary}
\newtheorem{lem}[thm]{Lemma}
\newtheorem{proposition}[thm]{Proposition}
\newtheorem{example}[thm]{Example}
\theoremstyle{definition}
\newtheorem{definition}[thm]{Definition}
\theoremstyle{remark}
\newtheorem{remark}[thm]{Remark}
\numberwithin{equation}{section}

\begin{document}
\date{}
\title{On Hom-Lie superbialgebras}
\author{Mohamed Fadous, Sami Mabrouk, Abdenacer Makhlouf }
 \maketitle{}


 \begin{abstract}
The purpose  of this paper is to generalize  to $\mathbb{Z}_2$-graded case the study  of Hom-Lie bialgebras which were
discussed  first by D. Yau, then by C. Bai and Y. Sheng. We provide different ways for constructing Hom-Lie superbialgebras. Also we define
Matched pairs, Manin supertriples and discuss their relationships. Moreover, we  study coboundary  and
triangular Hom-Lie bialgebras, as well as infinitesimal deformations of the cobracket.
\end{abstract}
\begin{small}
\textbf{ Keywords :}{Hom-Lie superalgebra, coboundary   Hom-Lie bialgebra, 
triangular Hom-Lie bialgebra}\\
\textbf{2010 Mathematics Subject Classification : } 17B62, 17B37.
\end{small}
\section*{Introduction}

Lie bialgebras were introduced by V. Drinfel'd  in  \cite{Drinfel'd V.G.2, Drinfel'd V.G.4}, they are infinitesimal versions of compatible Poisson structures on Lie groups and maybe viewed as the Lie-theoretic case of a bialgebra. He raised various problems related to quantum groups and quantization.  The study of quasi-triangular quantum groups involves the solutions of the quantum
Yang-Baxter equations. In the classical limit, the solutions of the classical Yang-Baxter equations
provide examples of Lie bialgebras. Since then a huge research activity was dedicated to these kind of algebraic structures.

The aim of this paper is to define and study Hom-Lie superbialgebras which are Hom-type generalization of Lie superbialgebras. Hom-Lie superbialgebras are Hom-Lie superalgebras provided with a cobracket and a compatibility condition. Motivated by examples of $q$-deformations of algebras of vector fields,   J. Hartwig, D. Larsson, and S. Silvestrov introduced the notion of  Hom-Lie algebra  in \cite{HartwigLarssonSilvestrov}, as a generalization of  Lie algebras where the Jacobi condition is twisted by a Homomorphism. The graded case of Hom-Lie algebras were studied first by F. Ammar and the last author  in \cite{Ammar F Makhlouf A 1}, while Hom-Lie bialgebras  were
discussed  by D. Yau, then by C. Bai and Y. Sheng. Recently L. Cai and Y. Sheng presented a slightly different approach of Hom-Lie bialgebras called purely Hom-Lie bialgebras in \cite{Sheng2}.

The paper is organized as follows, in the first section we provide the relevant definitions and some properties about Hom-Lie superbialgebras. Moreover, we give some key constructions and a classification  of 3-dimensional Hom-Lie superbialgebras with 2-dimensional even part. In Section 2, We define  Matched pairs and Manin supertriples, then we establish their relationships with  Hom-Lie superbialgebras. We construct a Hom-Lie superalgebra structure on the direct sum of two Hom-Lie superalgebras $(\mathfrak{g},[\cdot,\cdot],\phi)$ and $(\mathfrak{g}',[\cdot,\cdot]',\phi')$, such that $\mathfrak{g}$ is a $\mathfrak{g}'$-module and  $\mathfrak{g}'$ is a $\mathfrak{g}$-module, also we construct a  Hom-Lie superbialgebra structure on the direct sum $\mathfrak{g}\oplus\mathfrak{g}^*$ where  $\mathfrak{g}^* $ is the dual superspace of $\mathfrak{g}$. Section 3 is dedicated to coboundary Hom-Lie superbialgebras and quasi-triangular Hom-Lie superbialgebras. We show how a coboundary or quasi-triangular Hom-Lie superbialgebra can be constructed from a Hom-Lie superalgebra and an $r$-matrix. In the last section, we study perturbation of cobrackets in Hom-Lie superbialgebras, following Drinfel'd's perturbation theory of quasi-Hopf algebras. We describe Hom-Lie superbialgebras obtained by infinitesimal deformations of the cobracket.

\section{\bf Basics and Classification of Hom-Lie superbialgebras}
In this section, we introduce and study Hom-Lie superbialgebras, which are  Hom-type  version of Lie superbialgebras, see \cite{Drinfel'd V.G.2,Drinfel'd V.G.4}. We extend to graded case  the definition of Hom-Lie bialgebra introduced  in \cite{Yau2}. We show that the dual of a finite dimensional
Hom-Lie superbialgebra is also a Hom-Lie superbialgebra (Theorem \ref{amineeeee}), generalizing the
self-dual property of Hom-Lie bialgebras.

First, let us start by fixing some definitions and notations.
 Let $\mathfrak{L}=\mathfrak{L}_{\bar{0}}\oplus\mathfrak{L}_{\bar{1}}$ be a  $\mathbb{Z}_2$-graded vector space over an arbitrary field $\mathbb{K}$ of characteristic 0. In the sequel, we will consider only element which are  $\mathbb{Z}_{2}$-homogeneous. For $x\in\mathfrak{L}$, we denote by ${|x|}\in\mathbb{Z}_{2}$  its parity, i.e., $x\in \mathfrak{L}_{|x|}$.
We denote by $\tau$ the super-twist map of $\mathfrak{L}\otimes\mathfrak{L}$ namely,
$
 \tau(x\otimes y)=(-1)^{{|x|}{|y|}}y\otimes x
$
for $x,y\in \mathfrak{L}.$
The super-cyclic map $\xi$ permutes the coordinates of $\mathfrak{L}\otimes\mathfrak{L}\otimes\mathfrak{L}$, it  is defined as 
\begin{equation*}
\xi=(\mathbf{1}\otimes\tau)\cdot(\tau\otimes\mathbf{1}) : x\otimes y\otimes z\mapsto (-1)^{{|x|}({|y|}+{|z|})}y\otimes z\otimes x,
\end{equation*}
for $x, y, z\in\mathfrak{L}$, where $\mathbf{1}$ is the identity map on $\mathfrak{L} $. We denote by $\mathfrak{L^\ast}=$Hom$(\mathfrak{L}, \mathbb{K})$ the linear dual of $\mathfrak{L}$. For $\phi\in\mathfrak{L^\ast}$ and $x\in\mathfrak{L}$, we often use the adjoint notation $\langle\phi, x\rangle$ for $\phi(x)\in\mathbb{K}$.\\
For a linear map $\Delta: \mathfrak{L}\rightarrow \mathfrak{L}\otimes\mathfrak{L}$ (comultiplication), we use Sweedler's notation $\Delta(x)=\sum_{(x)} x_{1}\otimes x_{2}$ for $x\in\mathfrak{L}$.
We will often omit the summation sign $\sum_{(x)}$ to make it simple.
The parity $|r|$ of $r\in\mathfrak{L}^{\otimes2}$ is defined as follows : since we assume $r$  homogenous, there exists $|r|\in\mathbb{Z}_{2}$, such that $r$ can be written as $r=\sum r_{1}\otimes r_{2}\in\mathfrak{L}^{\otimes2}$, $r_{1}, r_{2}$ are homogenous elements with $|r|=|r_{1}|+|r_{2}|$.

\begin{definition}(\cite{Ammar F Makhlouf A 1}). A \emph{Hom-Lie superalgebra } is a triple $(\mathfrak{L}, [\cdot ,\cdot ], \alpha)$ consisting of a superspace $\mathfrak{L}$, an even bilinear map $[\cdot ,\cdot ]:\mathfrak{L} \times \mathfrak{L} \rightarrow \mathfrak{L}$ and an even superspace homomorphism $\alpha:\mathfrak{L} \rightarrow \mathfrak{L}$ satisfying
\begin{equation}\label{701}
[x,y]=-(-1)^{{|x|}{|y|}}[y,x],
\end{equation}
\begin{equation}\label{702}
(-1)^{{|x|}{|z|}}[\alpha(x),[y,z]]+(-1)^{{|z|}{|y|}}[\alpha(z),[x,y]]+(-1)^{{|y|}{|x|}}[\alpha(y),[z,x]]=0
\end{equation}
for all homogeneous elements $x,y,z$ in $\mathfrak{L}$.\\
It is multiplicative if, in addition $\alpha\circ[\cdot ,\cdot ]=[\cdot ,\cdot ]\circ\alpha^{\otimes2}$, (i.e., $\alpha([x, y])=[\alpha(x), \alpha(y)]$,  $\forall x,y \in \mathfrak{L}$).

\end{definition}
 \begin{definition}\label{00001} (\cite{Hengyun Y and Yucai S}, \cite{MakhloufSilvestrov2}). A \emph{Hom-Lie supercoalgebra} is a triple $(\mathfrak{L}, \Delta, \alpha)$ consisting of a superspace $\mathfrak{L}$,  an even superspace homomorphism $\alpha:\mathfrak{L} \rightarrow \mathfrak{L}$ and a linear map  $\Delta:\mathfrak{L} \rightarrow \mathfrak{L}\otimes\mathfrak{L}$ (the cobracket) such that
\begin{equation}\label{2007}
\Delta(\mathfrak{L}^i) \subset \sum_{i=j+k} \mathfrak{L}^j\otimes\mathfrak{L}^k \ \  for \ \  i\in\mathbb{Z}_{2},
\end{equation}
\begin{equation}
 Im\Delta\subset Im(1\otimes1-\tau)  \ \  i.e. \ \ \Delta \ is \ skew-supersymmetric,
\end{equation}
\begin{equation}\label{jacobi}
(1\otimes1\otimes1+\xi+\xi^2)\circ(\alpha\otimes\Delta)\circ\Delta=0 : \mathfrak{L}\rightarrow \mathfrak{L}\otimes\mathfrak{L}\otimes\mathfrak{L}.
\end{equation}
If, in addition, $\Delta\circ\alpha=\alpha^{\otimes2}\circ\Delta$, then $\mathfrak{L}$ is called co-multiplicative.\\
A Hom-Lie supercoalgebra with $\alpha=Id$ is exactly a Lie supercoalgebra  \cite{Hengyun Y and Yucai S}.
\end{definition}
\begin{remark} Let $Im\Delta\subset Im(1\otimes1-\tau)\subset ker(1\otimes1+\tau)$, then $(1\otimes1+\tau)\circ\Delta=0$ \ \ i.e.\ \ $\Delta$ is skew-supersymmetric (\cite{Walter}).
\end{remark}
\begin{definition}
(1) For an element $x$ in a Hom-Lie superalgebra $(\mathfrak{L}, [\cdot ,\cdot ], \alpha)$ and $n\geq 2$, define the adjoint map  $ad_{x}:\mathfrak{L}^{\otimes n}\rightarrow \mathfrak{L}^{\otimes n}$ by
\begin{equation}\label{action}
ad_{x}(y_{1}\otimes\cdot\cdot\cdot\otimes y_{n})=\sum_{i=1}^{n} (-1)^{{|x|}({|y_{1}|+{|y_{2}|}}+\cdot\cdot\cdot+{|y_{i-1}|})}\alpha(y_{1})\otimes\cdot\cdot\cdot\otimes\alpha(y_{i-1})\otimes[x, y_{i}]\otimes\alpha(y_{i+1})\cdot\cdot\cdot\otimes\alpha(y_{n}).
\end{equation}
For  $n=2$, $ad_{x}(y_{1}\otimes y_{2})=[x, y_{1}]\otimes\alpha(y_{2})+(-1)^{{|x|}{|y_{1}|}}\alpha(y_{1})\otimes[x, y_{2}]$.\\
Conversely, given $\gamma=y_{1}\otimes\cdot\cdot\cdot\otimes y_{n}$, we define the map $ad(\gamma):\mathfrak{L}\rightarrow \mathfrak{L}^{\otimes n}$ by $ad(\gamma)(x)=ad_{x}(\gamma)$, for $x\in\mathfrak{L}$.\\
(2) For an element $x$ in a Hom-Lie superalgebra $(\mathfrak{L}, [\cdot ,\cdot ], \alpha)$ and $n\geq 2$, define the adjoint map \\ $ad_{\alpha(x)}:\mathfrak{L}^{\otimes n}\rightarrow \mathfrak{L}^{\otimes n}$ by
\begin{equation}
ad_{\alpha(x)}(y_{1}\otimes\cdot\cdot\cdot\otimes y_{n})=\sum_{i=1}^{n} (-1)^{{|x|}({|y_{1}|+{|y_{2}|}}+\cdot\cdot\cdot+{|y_{i-1}|})}\alpha(y_{1})\otimes\cdot\cdot\cdot\otimes\alpha(y_{i-1})\otimes[\alpha(x), y_{i}]\otimes\alpha(y_{i+1})\cdot\cdot\cdot\otimes\alpha(y_{n}).
\end{equation}
\end{definition}
\begin{definition}\label{ae}A \emph{(multiplicative) Hom-Lie superbialgebra } is a quadruple $(\mathfrak{L}, [\cdot ,\cdot ], \Delta, \alpha)$ such that
\begin{enumerate}
\item $(\mathfrak{L}, [\cdot ,\cdot ], \alpha)$ is a (multiplicative) Hom-Lie superalgebra.
\item $(\mathfrak{L}, \Delta, \alpha)$ is a (co-multiplicative) Hom-Lie supercoalgebra.
\item The following compatibility condition holds for all $x, y \in\mathfrak{L}$ :
\begin{equation}\label{a}
\Delta([x,y])=ad_{\alpha(x)}(\Delta(y))-(-1)^{{|x|}{|y|}}ad_{\alpha(y)}(\Delta(x)).
\end{equation}
\end{enumerate}
\end{definition}
\begin{definition}
 The map $f:\mathfrak{L} \rightarrow \mathfrak{L'}$ is called  even (resp. odd) map if $f(\mathfrak{L}_{i})\subset \mathfrak{L'}_{i}$ (resp. $f(\mathfrak{L}_{i})\subset \mathfrak{L'}_{i+1}),$ for $i=0, 1$.
A morphism of Hom-Lie superbialgebras is an even  linear map such that \\
$$\alpha\circ f=f\circ\alpha,\ \ \ \   f\circ[\cdot ,\cdot ]=[\cdot ,\cdot ]\circ f^{\otimes2}\ \ \ \      \text{and }\ \ \  \Delta\circ f=f^{\otimes2}\circ\Delta.$$\\
An isomorphism of Hom-Lie superbialgebras is an invertible morphism of Hom-Lie superbialgebras. Two Hom-Lie superbialgebras are said to be isomorphic if there exists an isomorphism between them.
\end{definition}
\begin{remark}
A Hom-Lie superbialgebra with $\alpha=Id$ is exactly a Lie superbialgebra, as defined in  \cite{Hengyun Y and Yucai S,Drinfel'd V.G.2,Drinfel'd V.G.4}. 
\end{remark}
\begin{remark}\label{2000}The compatibility condition (\ref{a}) is, in fact, a cocycle condition in Hom-Lie superalgebra cohomology  \cite{MakhloufSilvestrov3}, just as it is the case in a Lie superbialgebra with Lie superalgebra cohomology  \cite{Drinfel'd V.G.4}. Indeed, we can regard $\mathfrak{L}^{\otimes2}$ as an $\mathfrak{L}$-module via the $\alpha$-twisted adjoint action (\ref{a}):
\begin{align}\label{2001}
x\cdot(y_{1}\otimes y_{2})&=ad_{\alpha(x)}(y_{1}\otimes y_{2})\\
&=[\alpha(x), y_{1}]\otimes\alpha(y_{2})+(-1)^{{|x|}{|y_{1}|}}\alpha(y_{1})\otimes[\alpha(x), y_{2}],\nonumber\end{align}
for $x\in\mathfrak{L}$ and $y_{1}\otimes y_{2}\in\mathfrak{L}^{\otimes2}$.\\
 Then we can think of the cobracket $\Delta:\mathfrak{L} \rightarrow \mathfrak{L}\otimes\mathfrak{L}$ as a 1-cochain $\Delta\in C^1(\mathfrak{L},\mathfrak{L}^{\otimes2})$. Here $ C^1(\mathfrak{L},\mathfrak{L}^{\otimes2})$ is defined as the linear super-subspace of Hom$(\mathfrak{L},\mathfrak{L}^{\otimes2})$ consisting of maps that commute with $\alpha$. Generalizing \cite{MakhloufSilvestrov3} to include coefficients in $\mathfrak{L}^{\otimes2}$, the differential on $\Delta$ is given by
\begin{equation}\label{fff}
(\delta^{1}_{HL}\Delta)(x, y)=\Delta([x,y])-x\cdot\Delta(y)+(-1)^{{|x|}{|y|}}y\cdot\Delta(x)
=\Delta([x,y])-ad_{\alpha(x)}(\Delta(y))+(-1)^{{|x|}{|y|}}ad_{\alpha(y)}(\Delta(x)).
\end{equation}
Therefore, (\ref{a}) says exactly that $\Delta\in C^1(\mathfrak{L},\mathfrak{L}^{\otimes2})$ is a 1-cocycle.
\end{remark}

\begin{example} \emph{Classification of 2-dimensional Hom-Lie superbialgebras  with 1-dimensional even part}. \\
Let $\mathfrak{L}=\mathfrak{L}_{\bar{0}}\oplus\mathfrak{L}_{\bar{1}}$ be a 2-dimensional superspace where $\mathfrak{L}_{\bar{0}}$ is generated by $e_{1}$ and $\mathfrak{L}_{\bar{1}}$ is generated by $e_{2}$. The triple $(\mathfrak{L}, [\cdot ,\cdot ], \alpha)$ is a Hom-Lie superalgebra when  $[e_{1}, e_{1}]=0$, $[e_{1}, e_{2}]=b e_{2}$ and $[e_{2}, e_{2}]=c e_{1}$ with $\alpha(e_{1})=a_{1}e_{1}$, $\alpha(e_{2})=a_{2}e_{2}$ and  $a_{2}bc=0$, where $b,c, a_{1}, a_{2}$ are parameters in $\mathbb{K}$.\\
The triple $(\mathfrak{L}, \Delta, \alpha)$ is a Hom-Lie supercoalgebra if  $\Delta(e_{1})=0$ and $\Delta(e_{2})=d(e_{1}\otimes e_{2}-e_{2}\otimes e_{1})$,
where $d\in\mathbb{K}$.\\
The triple $(\mathfrak{L}, [\cdot ,\cdot ], \Delta, \alpha)$ is a Hom-Lie superbialgebras if $a_{1}=1$ or $a_{1}=-1$ and $a_{2}bd=0$.
\end{example}

\begin{remark}
Recently Cai and Sheng introduced a different notion of Hom-Lie bialgebras called \emph{purely Hom-Lie bialgebra}, see \cite{Sheng2}, which we can extend to the super case as follows.  Let $(\mathfrak{L}, [\cdot ,\cdot ], \alpha)$, where $\alpha$ is invertible, and $(\mathfrak{L}^*, [\cdot ,\cdot ], (\alpha^{-1})^*)$ be two Hom-Lie superalgebras. The pair $(\mathfrak{L},\mathfrak{L}^*) $ is a purely Hom-Lie superbialgebra if holds  the compatibility condition
\begin{equation}
\Delta([x,y])=ad_{\alpha^{-1}(x)}(\Delta(y))-(-1)^{{|x|}{|y|}}ad_{\alpha^{-1}(y)}(\Delta(x)).
\end{equation}
Notice that this condition is different from condition \eqref{a}. Purely Hom-Lie bialgebras in a graded case will be studied in a forthcoming paper.
\end{remark}

 \textbf{Classification of 3-dimensional Hom-Lie superbialgebras with 2-dimensional even part.} \\
Let $\mathfrak{L}=\mathfrak{L}_{\bar{0}}\oplus\mathfrak{L}_{\bar{1}}$ be a 3-dimensional superspace where $\mathfrak{L}_{\bar{0}}$ is generated by $e_{1}$, $e_{2}$ and $\mathfrak{L}_{\bar{1}}$ is generated by $e_{3}$. 
We aim to construct Hom-Lie superbialgebras . We set for  the linear map $\alpha$ 
$$\alpha(e_{1})=a_{1}e_{1}+a_{2}e_{2},\ \ \ \ \alpha(e_{2})=a_{3}e_{1}+a_{4}e_{2}, \ \ \ \ \ \alpha(e_{3})=a_{5}e_{3},$$
where $a_{1}, a_{2}, a_{3}, a_{4}, a_{5}$ are parameters in $\mathbb{K}.$\\
The structure of the bracket $[\cdot ,\cdot ]$ is of the form 
$$[e_{3},e_{3}]=b_{1}e_{1}+b_{2}e_{2},\ \ \ \ [e_{1},e_{2}]=b_{3}e_{1}+b_{4}e_{2},\ \ \ \ [e_{1},e_{3}]=b_{5}e_{3},\ \ \ \ [e_{2},e_{3}]=b_{6}e_{3},$$
where $b_{1}, b_{2}, b_{3}, b_{4}, b_{5}, b_{6}$ are parameters in $\mathbb{K}.$\\
The structure of the cobracket $\Delta$ is of the form 
\begin{eqnarray*}&\Delta(e_{1})=c_{1}(e_{1}\otimes e_{2}-e_{2}\otimes e_{1})+c_5 e_{3}\otimes e_{3} ,\ \ \ \Delta(e_{2})=c_{2}(e_{1}\otimes e_{2}-e_{2}\otimes e_{1})+c_6 e_{3}\otimes e_{3},\\  & \Delta(e_{3})=c_{3}(e_{1}\otimes e_{3}-e_{3}\otimes e_{1})+c_{4}(e_{2}\otimes e_{3}-e_{3}\otimes e_{2}),
\end{eqnarray*}
where $c_{1}, c_{2}, c_{3}, c_{4}$ are parameters in $\mathbb{K}.$\\

In the sequel, we will consider the linear map under the Jordan forms with respect to a suitable basis. We split the calculations in two cases. First, we deal  with diagonal case and then with Jordan case.\\

\textbf{1) Diagonal case }: We consider the  linear map $\alpha$, where with respect to a suitable basis the matrix is of the form:   
$\alpha=\left(\begin{array}{lll}
a_{1} \ \ 0 \ \ \ 0\\
0 \ \ \ a_{4} \ \ 0\\
0 \ \ \ 0 \ \ \ a_{5}\\
\end{array}\right)$, which corresponds to $a_{2}=0$, $a_{3}=0$, and  the eigenvalues are pairwise non-equal. \\
We obtained, when the eigenvalues are nonzero, the following corresponding  (multiplicative) Hom-Lie superbialgebras :\\

\begin{small}

\begin{tabular}{|l|l|l|}
  \hline
  \ \ \ \ \ \ \ \ \ \ Linear map &\ \ \ \ \ \ \ \ bracket &\ \ \ \ \ \ \ \ cobracket \\ \hline
   $\alpha(e_{1})=e_{1}, \ \ \alpha(e_{2})=a_{4}e_{2},\ \ \alpha(e_{3})=-e_{3}$ & $[e_{3}, e_{3}]=0, [e_{1}, e_{3}]=b_{5}e_{3} \ or \ (b_{5}=0),$& $\Delta(e_{1})=c_{5}e_{3}\otimes e_{3},$ \\
    & &$\Delta(e_{2})=0,$ \\
    &$[e_{1}, e_{2}]=b_{4}e_{2}, [e_{2}, e_{3}]=0$& $\Delta(e_{3})=0$ \\ \hline
  $\alpha(e_{1})=e_{1}, \ \ \alpha(e_{2})=a_{4}e_{2}, \ \ \alpha(e_{3})=a_{5}e_{3}$  & \ \ \  All bracket are zero&  $\Delta(e_{1})=0,$ \\
   &  &$\Delta(e_{2})=c_{2}(e_{1}\otimes e_{2}-e_{2}\otimes e_{1}),$ \\
    &&  $\Delta(e_{3})=c_{3}(e_{1}\otimes e_{3}-e_{3}\otimes e_{1})$ \\ \hline
      $\alpha(e_{1})=e_{1}, \ \ \alpha(e_{2})=a_{4}e_{2}, \ \alpha(e_{3})=-e_{3} $ & $[e_{3}, e_{3}]=b_{1}e_{1},  $& $\Delta(e_{1})=0,$  \\
    &$$ &$\Delta(e_{2})=c_{2}(e_{1}\otimes e_{2}-e_{2}\otimes e_{1}),$ \\
    &$[e_{1}, e_{2}]=[e_{1}, e_{3}]=[e_{2}, e_{3}]=0$& $\Delta(e_{3})=c_{3}(e_{1}\otimes e_{3}-e_{3}\otimes e_{1})$ \\
    \hline
      $\alpha(e_{1})=a_{1}e_{1}, \ \ \alpha(e_{2})=e_{2}, \ \alpha(e_{3})=-e_{3} $ & $[e_{3}, e_{3}]=[e_{1}, e_{3}]=0, [e_{1}, e_{2}]=b_{3}e_{1}$& $\Delta(e_{1})=0,$ \\
    & &$\Delta(e_{2})=c_{6}e_{3}\otimes e_{3},$ \\
    &$[e_{2}, e_{3}]=b_{6}e_{3} \ or \ (b_{6}=0 \ and \ b_{3}=0)$& $\Delta(e_{3})=0$ \\
   \hline
  $\alpha(e_{1})=a_{1}e_{1}, \ \ \alpha(e_{2})=e_{2}, \ \alpha(e_{3})=-e_{3} $ & $[e_{1}, e_{2}]=[e_{1}, e_{3}]=[e_{2}, e_{3}]=0,$& $\Delta(e_{1})=c_{1}(e_{1}\otimes e_{2}-e_{2}\otimes e_{1}),$ \\
    & &$\Delta(e_{2})=0,$ \\
    &$[e_{3}, e_{3}]=b_{2}e_{2} \ or \ (b_{2}=0)$& $\Delta(e_{3})=c_{4}(e_{2}\otimes e_{3}-e_{3}\otimes e_{2})$ \\
 \hline
$\alpha(e_{1})=e_{1}, \ \alpha(e_{2})=a_{4}e_{2}, \ \alpha(e_{3})=a_{5}e_{3} $ & $[e_{3}, e_{3}]=[e_{1}, e_{3}]=[e_{2}, e_{3}]=0, $& $\Delta(e_{1})=0,$ \\
    &$$&$\Delta(e_{2})=c_{2}(e_{1}\otimes e_{2}-e_{2}\otimes e_{1}), $ \\
    &$[e_{1}, e_{2}]=b_{4}e_{2}$&  $\Delta(e_{3})=0$  \\
     \hline
  $\alpha(e_{1})=e_{1},\ \alpha(e_{2})=a_{4}e_{2},\ \alpha(e_{3})=a_{5}e_{3}$ &  $[e_{3}, e_{3}]=[e_{2}, e_{3}]=0$ & All cobracket are zero \\
   $$
    & $[e_{1}, e_{2}]=b_{4}e_{2}, [e_{1}, e_{3}]=b_{5}e_{3}$ & \\
    \hline
   $\alpha(e_{1})=a_{1}e_{1}, \ \ \alpha(e_{2})=a_{4}e_{2}, \ \ \alpha(e_{3})=a_{5}e_{3}$  & \ \ \  All bracket are zero&  $\Delta(e_{1})=c_{5}e_{3}\otimes e_{3},$ \\
   $a_{5}=\pm\sqrt{a_{1}} \ or \ (\alpha(e_{1})=e_{1} \ and \ \alpha(e_{3})=-e_{3})$&  &$\Delta(e_{2})=0,$ \\
    &&  $\Delta(e_{3})=0$ \\
     \hline
  $\alpha(e_{1})=a_{1}e_{1},\ \alpha(e_{2})=e_{2},\ \alpha(e_{3})=a_{5}e_{3}$ &  $[e_{3}, e_{3}]=[e_{1}, e_{3}]=0$ & All cobracket are zero \\
   $$
    & $[e_{1}, e_{2}]=b_{3}e_{1}, [e_{2}, e_{3}]=b_{6}e_{3}$ & \\
    \hline
    $\alpha(e_{1})=e_{1}, \ \ \alpha(e_{2})=a_{4}e_{2}, \ \ \alpha(e_{3})=a_{5}e_{3}$ & $[e_{3}, e_{3}]=b_{2}e_{2},\ [e_{1}, e_{2}]=b_{4}e_{2},$& $\Delta(e_{1})=0, \Delta(e_{2})=c_{2}(e_{1}\otimes e_{2}-$ \\
 $(a_{5}=\pm \sqrt{a_{4}})$
    &$(\ b_{5}=\frac{a_{5}b_{4}}{2a_{4}}, c_{3}=\frac{a_{5}c_{2}}{2a_{4}}, c_{6}=-\frac{b_{4}c_{2}}{b_{2}})$ &$e_{2}\otimes e_{1})+c_{6}e_{3}\otimes e_{3}, $ \\
    &$[e_{2}, e_{3}]=0, \ [e_{1}, e_{3}]=b_{5}e_{3},$& $\Delta(e_{3})=c_{3}(e_{1}\otimes e_{3}-e_{3}\otimes e_{1})$ \\
    \hline
    $\alpha(e_{1})=a_{1}e_{1}, \ \ \alpha(e_{2})=e_{2}, \ \ \alpha(e_{3})=a_{5}e_{3}$ & $[e_{3}, e_{3}]=b_{1}e_{1},\ [e_{1}, e_{2}]=b_{3}e_{1},$& $\Delta(e_{1})=c_{1}(e_{1}\otimes e_{2}-e_{2}\otimes e_{1})$ \\
 $(a_{5}=\pm \sqrt{a_{1}})$
    &$(\ b_{6}=-\frac{a_{5}b_{3}}{2a_{1}}, c_{4}=-\frac{a_{5}c_{1}}{2a_{1}}, c_{5}=-\frac{b_{3}c_{1}}{b_{1}})$ &$+c_{5}e_{3}\otimes e_{3}, \Delta(e_{2})=0, $ \\
    &$[e_{1}, e_{3}]=0, \ [e_{2}, e_{3}]=b_{6}e_{3},$& $\Delta(e_{3})=c_{4}(e_{2}\otimes e_{3}-e_{3}\otimes e_{2})$ \\
    \hline
    $\alpha(e_{1})=a_{1}e_{1}, \ \ \alpha(e_{2})=e_{2}, \ \ \alpha(e_{3})=a_{5}e_{3}$ & $[e_{3}, e_{3}]=0,\ [e_{1}, e_{2}]=b_{3}e_{1},$& $\Delta(e_{1})=c_{5}e_{3}\otimes e_{3},$ \\
 $(a_{5}=\pm \sqrt{a_{1}})$
    &$(\ b_{6}=-\frac{a_{5}b_{3}}{2a_{1}})$ &$\Delta(e_{2})=0, $ \\
    &$[e_{1}, e_{3}]=0, \ [e_{2}, e_{3}]=b_{6}e_{3},$& $\Delta(e_{3})=0$ \\
    \hline
  $\alpha(e_{1})=a_{1}e_{1}, \ \ \alpha(e_{2})=e_{2}, \ \ \alpha(e_{3})=a_{5}e_{3}$  & \ \ \  All bracket are zero  &   $\Delta(e_{1})=c_{1}(e_{1}\otimes e_{2}-e_{2}\otimes e_{1})$ \\
   $(a_{5}=\pm\sqrt{a_{1}}, c_{4}=-\frac{a_{5}c_{1}}{2a_{1}})$&  &$+c_{5}e_{3}\otimes e_{3}, \Delta(e_{2})=0,$ \\
    &&  $\Delta(e_{3})=c_{4}(e_{2}\otimes e_{3}-e_{3}\otimes e_{2})$ \\
    \hline
         \end{tabular}\\

\end{small}

\begin{small}
\begin{tabular}{|l|l|l|}
    \hline
   $\alpha(e_{1})=e_{1}, \ \ \alpha(e_{2})=a_{4}e_{2}, \ \ \alpha(e_{3})=a_{5}e_{3}$  & \ \ \  All bracket are zero  &   $\Delta(e_{1})=0, \Delta(e_{2})=c_{2}(e_{1}\otimes e_{2}$ \\
   $(a_{5}=\pm\sqrt{a_{4}}, c_{3}=\frac{a_{5}c_{2}}{2a_{4}})$&  &$-e_{2}\otimes e_{1})+c_{6}e_{3}\otimes e_{3},$ \\
    &&  $\Delta(e_{3})=c_{3}(e_{1}\otimes e_{3}-e_{3}\otimes e_{1})$ \\
    \hline
    $\alpha(e_{1})=e_{1}, \ \ \alpha(e_{2})=a_{4}e_{2}, \ \ \alpha(e_{3})=a_{5}e_{3}$  & \ \ \  All bracket are zero  &   $\Delta(e_{1})=0, $ \\
   $(a_{5}=\pm\sqrt{a_{4}})$&  &$\Delta(e_{2})=c_{2}(e_{1}\otimes e_{2}-e_{2}\otimes e_{1}),$ \\
    &&  $\Delta(e_{3})=c_{3}(e_{1}\otimes e_{3}-e_{3}\otimes e_{1})$ \\
     \hline
    $\alpha(e_{1})=e_{1}, \ \alpha(e_{2})=a_{4}e_{2}, \ \alpha(e_{3})=a_{5}e_{3} $ & $[e_{1}, e_{2}]=b_{4}e_{2},$& $\Delta(e_{1})=0,$ \\
    $(a_{5}=\pm\sqrt{a_{4}})$&&$\Delta(e_{2})=c_{2}(e_{1}\otimes e_{2}-e_{2}\otimes e_{1}), $ \\
    &$[e_{3}, e_{3}]=[e_{1}, e_{3}]=[e_{2}, e_{3}]=0$&  $\Delta(e_{3})=0$ \\
  \hline
    $\alpha(e_{1})=e_{1}, \ \ \alpha(e_{2})=a_{4}e_{2}, \ \ \alpha(e_{3})=a_{5}e_{3}$ & $[e_{1}, e_{2}]=b_{4}e_{2},\ [e_{1}, e_{3}]=b_{5}e_{3},$& $\Delta(e_{1})=0,$ \\
 $(a_{5}=\pm \sqrt{a_{4}})$
    &$(\ b_{5}=\frac{a_{5}b_{4}}{2a_{4}})$ &$ \Delta(e_{2})=c_{6}e_{3}\otimes e_{3}, $ \\
    &$[e_{2}, e_{3}]=[e_{3}, e_{3}]=0,$& $\Delta(e_{3})=0$ \\
 \hline
 $\alpha(e_{1})=a_{1}e_{1},\ \alpha(e_{2})=a_{4}e_{2},\  \alpha(e_{3})=a_{5}e_{3}$ &\ \ \  All bracket are zero &  All cobracket are zero \\
     &  &  \\
   \hline
   $\alpha(e_{1})=a_{1}e_{1}, \ \alpha(e_{2})=e_{2}, \ \alpha(e_{3})=a_{5}e_{3} $ & $[e_{1}, e_{2}]=b_{3}e_{1},$& $\Delta(e_{1})=c_{1}(e_{1}\otimes e_{2}-e_{2}\otimes e_{1}),$ \\
    $(a_{5}=\pm\sqrt{a_{1}})$&&$\Delta(e_{2})=0, $ \\
    &$[e_{1}, e_{3}]=[e_{2}, e_{3}]=[e_{3}, e_{3}]=0$&  $\Delta(e_{3})=0$ \\
  \hline
  $\alpha(e_{1})=a_{1}e_{1},\ \alpha(e_{2})=a_{4}e_{2},\ \alpha(e_{3})=a_{5}e_{3}$ &  $[e_{3}, e_{3}]=b_{1}e_{1},$ & All cobracket are zero \\
   $(a_{5}=\pm\sqrt{a_{1}})$
    & $[e_{1}, e_{2}]=[e_{1}, e_{3}]=[e_{2}, e_{3}]=0$ & \\
 \hline
 $\alpha(e_{1})=a_{1}e_{1},\ \alpha(e_{2})=a_{4}e_{2},\ \alpha(e_{3})=a_{5}e_{3}$ &  $[e_{3}, e_{3}]=b_{2}e_{2},$ & All cobracket are zero \\
   $(a_{5}=\pm\sqrt{a_{4}})$
    & $[e_{1}, e_{2}]=[e_{1}, e_{3}]=[e_{2}, e_{3}]=0$ & \\
  \hline
    $\alpha(e_{1})=e_{1}, \ \ \alpha(e_{2})=a_{4}e_{2}, \ \ \alpha(e_{3})=-e_{3}$  & \ \ \  All bracket are zero  &   $\Delta(e_{1})=0, $ \\
   $$&  &$\Delta(e_{2})=0,$ \\
    &&  $\Delta(e_{3})=c_{3}(e_{1}\otimes e_{3}-e_{3}\otimes e_{1})$\\ \hline
  $\alpha(e_{1})=a_{1}e_{1}, \ \ \alpha(e_{2})=e_{2}, \ \ \alpha(e_{3})=-e_{3}$  & \ \ \  All bracket are zero  &   $\Delta(e_{1})=0, $ \\
   $$&  &$\Delta(e_{2})=0,$ \\
    &&  $\Delta(e_{3})=c_{4}(e_{2}\otimes e_{3}-e_{3}\otimes e_{2})$\\ \hline
   $\alpha(e_{1})=e_{1}, \ \ \alpha(e_{2})=a_{4}e_{2},\ \ \alpha(e_{3})=-e_{3}$ & $[e_{1}, e_{2}]=b_{4}e_{2}, $& $\Delta(e_{1})=0,$ \\
    & &$\Delta(e_{2})=c_{2}(e_{1}\otimes e_{2}-e_{2}\otimes e_{1}),$ \\
    &$[e_{1}, e_{3}]=[e_{2}, e_{3}]=[e_{3}, e_{3}]=0$& $\Delta(e_{3})=0$ \\ \hline
\end{tabular}\\
\end{small}

\textbf{2) Jordan case} : Now, we consider the linear map $\alpha$  where the corresponding matrix is  of the form $\alpha=\left(\begin{array}{lll}
a_{1} \ \ 1 \ \ \ 0\\
0 \ \ \ a_{1} \ \ 0\\
0 \ \ \ 0 \ \ \ a_{5}\\
\end{array}\right)$, that is  $a_{2}=1$, $a_{3}=0$, $a_{4}=a_{1}$.\\
We obtained the following corresponding  (multiplicative) Hom-Lie superbialgebras :\\

\begin{small}
\begin{tabular}{|l|l|l|}
  \hline
  \ \ \ \ \ \ \ \ \ \ Linear map &\ \ \ \ \ \ \ \ bracket &\ \ \ \ \ \ \ \ cobracket \\ \hline
   $\alpha(e_{1})=e_{2}, \ \ \alpha(e_{2})=\alpha(e_{3})=0$ & $[e_{3}, e_{3}]=b_{2}e_{2},\ [e_{1}, e_{2}]=b_{4}e_{2},$& $\Delta(e_{1})=c_{1}(e_{1}\otimes e_{2}-e_{2}\otimes e_{1})$ \\
    & &$+c_{5} e_{3}\otimes e_{3}, \Delta(e_{2})=0,$ \\
    &$[e_{1}, e_{3}]=[e_{2}, e_{3}]=0$& $\Delta(e_{3})=c_{4}(e_{2}\otimes e_{3}-e_{3}\otimes e_{2})$ \\
    \hline
  $\alpha(e_{1})=e_{2}, \ \ \alpha(e_{2})=\alpha(e_{3})=0$  & $[e_{3}, e_{3}]=b_{2}e_{2},\ [e_{1}, e_{2}]=b_{4}e_{2}, $& $\Delta(e_{1})=c_{1}(e_{1}\otimes e_{2}-e_{2}\otimes e_{1})$ \\
    & &$+c_{5} e_{3}\otimes e_{3}, \Delta(e_{2})=0,$ \\
    &$[e_{1}, e_{3}]=b_{5}e_{3}, [e_{2}, e_{3}]=0, \ or \ (b_{5}=0)$& $\Delta(e_{3})=0$ \\
    \hline
  $\alpha(e_{1})=e_{1}+e_{2}, \ \alpha(e_{2})=e_{2}, \ \ \alpha(e_{3})=0$  & $[e_{3}, e_{3}]=[e_{2}, e_{3}]=0,$& $\Delta(e_{1})=c_{1}(e_{1}\otimes e_{2}-e_{2}\otimes e_{1}),$ \\
    & &$\Delta(e_{2})=0,$ \\
    &$[e_{1}, e_{2}]=b_{4}e_{2}, [e_{1}, e_{3}]=b_{5}e_{3} \ or \ (b_{5}=0)$& $\Delta(e_{3})=c_{4}(e_{2}\otimes e_{3}-e_{3}\otimes e_{2})$ \\
     \hline
      $\alpha(e_{1})=e_{2},\ \alpha(e_{2})=0,\ \alpha(e_{3})=a_{5}e_{3},$ & $[e_{3}, e_{3}]= [e_{1}, e_{3}]=[e_{2}, e_{3}]=0,$ &  $\Delta(e_{1})=c_{1}(e_{1}\otimes e_{2}-e_{2}\otimes e_{1}),$ \\
    $(a_{5}\neq0)$ & $[e_{1}, e_{2}]=b_{4}e_{2}$ & $\Delta(e_{2})=\Delta(e_{3})=0$ \\
  \hline
  $\alpha(e_{1})=e_{1}+e_{2}, \ \ \alpha(e_{2})=e_{2}, \ \alpha(e_{3})=a_{5}e_{3} $ &\ \ \  All bracket are zero& $\Delta(e_{1})=c_{1}(e_{1}\otimes e_{2}-e_{2}\otimes e_{1}),$ \\
    $$& &$\Delta(e_{2})=0,$ \\
    && $\Delta(e_{3})=c_{4}(e_{2}\otimes e_{3}-e_{3}\otimes e_{2})$ \\
 \hline
$\alpha(e_{1})=e_{1}+e_{2}, \ \alpha(e_{2})=e_{2}, \ \alpha(e_{3})=a_{5}e_{3} $ & $[e_{1}, e_{2}]=b_{4}e_{2},\ [e_{1}, e_{3}]=b_{5}e_{3},$& $\Delta(e_{1})=c_{1}(e_{1}\otimes e_{2}-e_{2}\otimes e_{1}),$ \\ $$
    &$$&$\Delta(e_{2})=0, \ for\ c_{4}=\frac{b_{5}c_{1}}{b_{4}}$ \\
    &$[e_{3}, e_{3}]=[e_{2}, e_{3}]=0$&  $\Delta(e_{3})=c_{4}(e_{2}\otimes e_{3}-e_{3}\otimes e_{2})$  \\
    \hline
    $\alpha(e_{1})=e_{1}+e_{2}, \ \ \alpha(e_{2})=e_{2}, \ \alpha(e_{3})=a_{5}e_{3} $ &\ \ \  All bracket are zero& $\Delta(e_{1})=c_{1}(e_{1}\otimes e_{2}-e_{2}\otimes e_{1})$ \\
    $(a_{5}=\pm1, c_{4}=-\frac{a_{5}c_{1}}{2})$& &$+c_{5}e_{3}\otimes e_{3}, \Delta(e_{2})=0,$ \\
    && $\Delta(e_{3})=c_{4}(e_{2}\otimes e_{3}-e_{3}\otimes e_{2})$ \\
   \hline
          \end{tabular}
   \end{small}

\begin{small}
\begin{tabular}{|l|l|l|}
\hline
   $\alpha(e_{1})=e_{1}+e_{2}, \ \ \alpha(e_{2})=e_{2}, \ \ \alpha(e_{3})=a_{5}e_{3}$ & $[e_{3}, e_{3}]=[e_{2}, e_{3}]=0,$& $\Delta(e_{1})=c_{1}(e_{1}\otimes e_{2}-e_{2}\otimes e_{1})$ \\
   $(a_{5}=\pm1, c_{4}=-\frac{a_{5}c_{1}}{2})$
    &$(\ b_{5}=-\frac{a_{5}b_{4}}{2})$ &$+c_{5}e_{3}\otimes e_{3}, \Delta(e_{2})=0, $ \\
    &$[e_{1}, e_{2}]=b_{4}e_{2}, [e_{1}, e_{3}]=b_{5}e_{3},$& $\Delta(e_{3})=c_{4}(e_{2}\otimes e_{3}-e_{3}\otimes e_{2})$ \\
     \hline
   $\alpha(e_{1})=e_{1}+e_{2}, \ \ \alpha(e_{2})=e_{2}, \ \ \alpha(e_{3})=a_{5}e_{3}$ & $[e_{1}, e_{2}]=[e_{1}, e_{3}]=[e_{2}, e_{3}]=0,$& $\Delta(e_{1})=c_{1}(e_{1}\otimes e_{2}-e_{2}\otimes e_{1}),$ \\
   $(a_{5}=\pm1)$
    &$$ &$\Delta(e_{2})=0, $ \\
    &$[e_{3}, e_{3}]=b_{2}e_{2}$& $\Delta(e_{3})=c_{4}(e_{2}\otimes e_{3}-e_{3}\otimes e_{2})$ \\
   \hline
 $\alpha(e_{1})=-e_{1}+e_{2},\ \alpha(e_{2})=-e_{2},\ \alpha(e_{3})=0,$ & $[e_{3}, e_{3}]=[e_{1}, e_{2}]=[e_{2}, e_{3}]=0,$ &  $\Delta(e_{1})=\Delta(e_{2})=0,$ \\
    $$ & $[e_{1}, e_{3}]=b_{5}e_{3}, \ or \ (b_{5}=0)$ & $\Delta(e_{3})=c_{4}(e_{2}\otimes e_{3}-e_{3}\otimes e_{2})$ \\
     \hline
    $\alpha(e_{1})=a_{1}e_{1}+e_{2},\ \alpha(e_{2})=a_{4}e_{2},\  \alpha(e_{3})=a_{5}e_{3}$ &\ \ \  All bracket are zero &  All cobracket are zero \\
    $(a_{1}\neq0, \ a_{5}\neq0)$ &  &  \\
   \hline
   $\alpha(e_{1})=e_{1}+e_{2}, \ \ \alpha(e_{2})=e_{2}, \ \ \alpha(e_{3})=a_{5}e_{3}$ & $[e_{3}, e_{3}]=[e_{2}, e_{3}]=0,$& $\Delta(e_{1})=c_{5}e_{3}\otimes e_{3},$ \\
   $(a_{5}=\pm1)$
    &$$ &$$ \\
    &$[e_{1}, e_{3}]=b_{5}e_{3}, [e_{1}, e_{2}]=b_{4}e_{2}, \ or \ (b_{4}=0)$& $\Delta(e_{2})=\Delta(e_{3})=0$ \\
  \hline
  $\alpha(e_{1})=e_{1}+e_{2}, \ \ \alpha(e_{2})=e_{2}, \ \ \alpha(e_{3})=a_{5}e_{3}$ & $[e_{3}, e_{3}]=b_{2}e_{2}, [e_{1}, e_{2}]=b_{4}e_{2},$& $\Delta(e_{1})=c_{1}(e_{1}\otimes e_{2}-e_{2}\otimes e_{1})$ \\
   $(a_{5}=\pm1)$
    &$(b_{5}=\frac{a_{5}b_{4}}{2}, c_{5}=\frac{-b_{4}c_{1}+2a_{5}b_{4}c_{4}}{2b_{2}}, c_{4}=\pm \frac{c_{1}}{2})$ &$+c_{5}e_{3}\otimes e_{3}, \Delta(e_{2})=0, $ \\
    &$[e_{1}, e_{3}]=b_{5}e_{3}, [e_{2}, e_{3}]=0,$& $\Delta(e_{3})=c_{4}(e_{2}\otimes e_{3}-e_{3}\otimes e_{2})$ \\
   \hline
  $\alpha(e_{1})=a_{1}e_{1}+e_{2}, \ \ \alpha(e_{2})=a_{4}e_{2}, \ \alpha(e_{3})=0 $ &\ \ \  All bracket are zero& $\Delta(e_{1})= \Delta(e_{2})=0,$ \\
    $(a_{1}=a_{4})$& &$\Delta(e_{3})=c_{4}(e_{2}\otimes e_{3}-e_{3}\otimes e_{2})$ \\
    && $$ \\
 \hline
 $\alpha(e_{1})=e_{1}+e_{2},\ \alpha(e_{2})=e_{2},\ \alpha(e_{3})=a_{5}e_{3}$ &  $[e_{1}, e_{3}]=b_{5}e_{3}, \ $ & $\Delta(e_{1})=\Delta(e_{2})=0,$ \\
   $(a_{5}=\pm1)$
    & $[e_{3}, e_{3}]=[e_{1}, e_{2}]=[e_{2}, e_{3}]=0$ & $\Delta(e_{3})=c_{4}(e_{2}\otimes e_{3}-e_{3}\otimes e_{2})$ \\
 \hline
 $\alpha(e_{1})=a_{1}e_{1}+e_{2},\ \alpha(e_{2})=a_{4}e_{2},\ \alpha(e_{3})=a_{5}e_{3}$ &  $[e_{3}, e_{3}]=b_{2}e_{2},$ & All cobracket are zero \\
   $(a_{5}=\pm\sqrt{a_{1}}, a_{1}=a_{4})$
    & $[e_{1}, e_{2}]=[e_{1}, e_{3}]=[e_{2}, e_{3}]=0$ & \\
 \hline
 $\alpha(e_{1})=a_{1}e_{1}+e_{2},\ \alpha(e_{2})=a_{4}e_{2},\ \alpha(e_{3})=0$ &  $[e_{1}, e_{3}]=b_{5}e_{3},$ & All cobracket are zero \\
   $(a_{1}=a_{4})$
    & $[e_{1}, e_{2}]=[e_{3}, e_{3}]=[e_{2}, e_{3}]=0$ & \\
 \hline
\end{tabular}\\
\end{small}

\

The following result shows that a Hom-Lie superbialgebra deforms into another Hom-Lie superbialgebra along any endomorphism.

\begin{thm}\label{aaa} Let $(\mathfrak{L}, [\cdot ,\cdot ], \Delta, \alpha)$ be a Hom-Lie superbialgebra and an even map $\beta:\mathfrak{L}\rightarrow\mathfrak{L}$ be a Hom-Lie superbialgebra morphism. Then
$$\mathfrak{L}_{\beta}=(\mathfrak{L},  [\cdot ,\cdot ]_{\beta}=\beta\circ[\cdot ,\cdot ], \Delta_{\beta}=\Delta\circ\beta,\beta\alpha)$$
is also a Hom-Lie superbialgebra, which is multiplicative if $\mathfrak{L}$ is.
\end{thm}
\begin{proof} It is immediate that $[\cdot ,\cdot ]_{\beta}$ is skew-supersymmetric (\ref{701}) because
$$ [x,y]_{\beta}=\beta([x,y])=\beta(-(-1)^{{|x|}{|y|}}[y,x])=-(-1)^{{|x|}{|y|}}\beta([y,x])=-(-1)^{{|x|}{|y|}}[y,x]_{\beta}.$$
The Hom super-Jacobi identity holds in $\mathfrak{L}_{\beta}$ because
$$(-1)^{{|x|}{|z|}} [\beta\alpha(x),[y,z]_{\beta}]_{\beta} =(-1)^{{|x|}{|z|}}\beta^{2}[\alpha(x),[y,z]]=-(-1)^{{|z|}{|y|}}[\beta\alpha(z),[x,y]_{\beta}]_{\beta}-(-1)^{{|y|}{|x|}}[\beta\alpha(y),[z,x]_{\beta}]_{\beta},$$ i.e.,
$$(-1)^{{|x|}{|z|}} [\beta\alpha(x),[y,z]_{\beta}]_{\beta}+(-1)^{{|z|}{|y|}}[\beta\alpha(z),[x,y]_{\beta}]_{\beta}+
(-1)^{{|y|}{|x|}}[\beta\alpha(y),[z,x]_{\beta}]_{\beta}=0.$$
$\Delta_{\beta}$ is skew-supersymmetric because
$$\tau\circ\Delta_{\beta}=\tau\circ\Delta\circ\beta=-\Delta\circ\beta=-\Delta_{\beta}.$$
Likewise, the Hom-super-co-jacobi identity holds in $\mathfrak{L}_{\beta}$ because
$$(1\otimes1\otimes1+\xi+\xi^{2})\circ(\beta\alpha\otimes\Delta_{\beta})\circ\Delta_{\beta}=(\beta^{\otimes{3}})^{2}(1\otimes1\otimes1+\xi+\xi^{2})\circ(\alpha\otimes\Delta)\circ\Delta=0.$$
To check the compatibility condition (\ref{a}) in $\mathfrak{L}_{\beta}$, we compute as follows :
\begin{align*}
\Delta_{\beta}([x,y]_{\beta})&=
(\beta^{\otimes{2}})^{2}\Delta([x,y])\\
&=(\beta^{\otimes{2}})^{2}([\alpha(x), y_{1}]\otimes\alpha(y_{2}))+(\beta^{\otimes{2}})^{2}((-1)^{{|x|}{|y_{1}|}}\alpha(y_{1})\otimes[\alpha(x), y_{2}])\\
&-(\beta^{\otimes{2}})^{2}((-1)^{{|x|}{|y|}}[\alpha(y), x_{1}]\otimes\alpha(x_{2}))-(\beta^{\otimes{2}})^{2}((-1)^{{|x|}{|y|}}(-1)^{{|y|}{|x_{1}|}}\alpha(x_{1})\otimes[\alpha(y), x_{2}])\\
&=[\beta\alpha(x),\beta(y_{1})]_{\beta}\otimes\beta\alpha(\beta(y_{2}))+(-1)^{{|x|}{|y_{1}|}}\beta\alpha(\beta(y_{1}))\otimes[\beta\alpha(x),\beta(y_{2})]_{\beta}\\
&-(-1)^{{|x|}{|y|}}[\beta\alpha(y),\beta(x_{1})]_{\beta}\otimes\beta\alpha(\beta(x_{2}))-(-1)^{{|x|}{|y|}}(-1)^{{|y|}{|x_{1}|}}\beta\alpha(\beta(x_{1}))\otimes[\beta\alpha(y),\beta(x_{2})]_{\beta}\\
&=[\beta\alpha(x),\beta(y_{1})]_{\beta}\otimes\beta\alpha(\beta(y_{2}))+(-1)^{{|x|}{|\beta(y_{1})|}}\beta\alpha(\beta(y_{1}))\otimes[\beta\alpha(x),\beta(y_{2})]_{\beta}\\
&-(-1)^{{|x|}{|y|}}[\beta\alpha(y),\beta(x_{1})]_{\beta}\otimes\beta\alpha(\beta(x_{2}))-(-1)^{{|x|}{|y|}}(-1)^{{|y|}{|\beta(x_{1})|}}\beta\alpha(\beta(x_{1}))\otimes[\beta\alpha(y),\beta(x_{2})]_{\beta}\\
&=ad_{\beta\alpha(x)}(\Delta_{\beta}(y))-(-1)^{{|x|}{|y|}}ad_{\beta\alpha(y)}(\Delta_{\beta}(x)).
\end{align*}
Because $|\beta(x_{1})|=|x_{1}|$, and $|\beta(y_{1})|=|y_{1}|$ (i.e., an even map $\beta$).
We have shown that $\mathfrak{L}_{\beta}$ is a Hom-Lie superbialgebra. The super-multiplicativity assertion is obvious.
\end{proof}
Now we discuss two special cases of Theorem \ref{aaa}. The next result says that one can obtain multiplicative Hom-Lie superbialgebras from Lie superbialgebras and their endomorphisms. A construction result of this form for Hom-type algebras was first given  in \cite{Yau1}.
\begin{cor}\label{aaaa} Let $(\mathfrak{L}, [\cdot ,\cdot ], \Delta)$ be a Lie superbialgebra and an even map $\beta:\mathfrak{L}\rightarrow\mathfrak{L}$ be a Lie superbialgebra morphism. Then
$$ \mathfrak{L}_{\beta}=(\mathfrak{L},  [\cdot ,\cdot ]_{\beta}=\beta\circ[\cdot ,\cdot ], \Delta_{\beta}=\Delta\circ\beta,\beta)$$
is a multiplicative Hom-Lie superbialgebra.
\end{cor}
\begin{proof} This is the $\alpha=Id$ special case of Theorem \ref{aaa}.
\end{proof}
The next result says that every multiplicative Hom-Lie superbialgebra gives rise to an infinite sequence of multiplicative Hom-Lie superbialgebras.
\begin{cor}\label{www} Let $(\mathfrak{L}, [\cdot ,\cdot ], \Delta, \alpha)$ be a multiplicative Hom-Lie superbialgebra. Then
$$ \mathfrak{L}_{\alpha^{n}}=(\mathfrak{L},  [\cdot ,\cdot ]_{\alpha^{n}}=\alpha^{n}\circ[\cdot ,\cdot ], \Delta_{\alpha^{n}}=\Delta\circ\alpha^{n},\alpha^{n+1})$$
is also a multiplicative Hom-Lie superbialgebra for each integer $n\geq0$.
\end{cor}
\begin{proof} This is the $\beta=\alpha^{n}$ special case of Theorem \ref{aaa}.
\end{proof}
Next we consider when Hom-Lie superbialgebra of the from $ \mathfrak{L}_{\beta}$, as in Corollary \ref{aaaa}, are isomorphic.
\begin{thm} \label{bb} Let $\mathfrak{g}$ and $\mathfrak{h}$ be Lie superbialgebras. Let  $\alpha:\mathfrak{g}\rightarrow\mathfrak{g}$ and $\beta:\mathfrak{h}\rightarrow\mathfrak{h}$ be Lie superbialgebras morphisms with $\beta$ and $\beta^{\otimes{2}}$ injective. Then the following statements are equivalent :\\
1) The Hom-Lie superbialgebras $\mathfrak{g}_{\alpha}$ and $\mathfrak{h}_{\beta}$ as in Corollary \ref{aaaa}, are isomorphic.\\
2) There exists a Lie superbialgebra isomorphism $\gamma:\mathfrak{g}\rightarrow\mathfrak{h}$ such that $\gamma\alpha=\beta\gamma.$
\end{thm}
\begin{proof} To show that the first statement implies the second statement, suppose that $\gamma:\mathfrak{g}_{\alpha}\rightarrow\mathfrak{h}_{\beta}$ is an isomorphism of Hom-Lie superbialgebras. Then $\gamma\alpha=\beta\gamma$ automatically.\\
To see that $\gamma$ is a Lie superbialgebra isomorphism, first we check that it commutes with the Lie bracket. For any two elements $x$ and $y$ in $\mathfrak{g}$, we have
$$\beta\gamma[x,y]=\gamma\alpha[x,y]
 =\gamma([x,y]_{\alpha})
 =[\gamma(x),\gamma(y)]_{\beta}
 =\beta[\gamma(x),\gamma(y)].$$
Since $\beta$ is injective, we conclude that $\gamma[x,y]=[\gamma(x),\gamma(y)],$
i.e., $\gamma$ is a Lie superbialgebra isomorphism.\\
To check that $\gamma$ commutes with the Lie cobrackets, we compute as follows:
\begin{align*}
\beta^{\otimes{2}}(\gamma^{\otimes{2}}(\Delta(x)))&=(\beta\gamma)^{\otimes{2}}(\Delta(x))
 =(\gamma\alpha)^{\otimes{2}}(\Delta(x))
 =\gamma^{\otimes{2}}(\alpha^{\otimes{2}}(\Delta(x)))
 =\gamma^{\otimes{2}}(\Delta_{\alpha}(x))
 =\Delta_{\beta}(\gamma(x))\\
&=\beta^{\otimes{2}}(\Delta(\gamma(x))).
\end{align*}
The injectivity of $\beta^{\otimes{2}}$ now implies that $\gamma$ commutes with the Lie cobrackets. Therefore, $\gamma$ is a Lie superbialgebra isomorphism. 
The other implication is proved by a similar argument, much of which is already given above.
\end{proof}
For a Lie superbialgebra $\mathfrak{g}$,  let Aut$(\mathfrak{g})$ be the group of Lie superbialgebra isomorphisms from $\mathfrak{g}$ to $\mathfrak{g}$. In Theorem \ref{bb}, restricting to the case $\mathfrak{g}=\mathfrak{h}$ with $\alpha$ and $\beta$ both invertible, we obtain the following special case.
\begin{cor}\label{bbb}Let $\mathfrak{g}$ be a Lie superbialgebra and $\alpha, \beta\in$ Aut$(\mathfrak{g})$. Then the Hom-Lie superbialgebras $\mathfrak{g}_{\alpha}$ and $\mathfrak{g}_{\beta}$, as in Corollary \ref{aaaa}, are isomorphic if and only if $\alpha$ and $\beta$ are conjugate in Aut$(\mathfrak{g})$.
\end{cor}
Corollary \ref{bbb} can be restated as follows.
\begin{cor} Let $\mathfrak{g}$ be a Lie superbialgebra.Then there is a bijection between the following two sets:\\
1) The set of isomorphism classes of Hom-Lie superbialgebras $\mathfrak{g}_{\alpha}$ with $\alpha$ invertible.\\
2) The set of conjugacy classes in the group Aut$(\mathfrak{g})$.
\end{cor}
The next result shows that finite dimensional Hom-Lie superbialgebras, like Lie superbialgebras, can be dualized. A proof of this self-dual property for the special of Lie bialgebras can be found in \cite{Majid}.
\begin{remark}\label{01}
$\bullet$ If $(\mathfrak{L}, \Delta, \alpha)$ is a Hom-Lie supercoalgebra, then $(\mathfrak{L}^{\ast}, [\cdot ,\cdot ] , \alpha)$ is a Hom-Lie superalgebra. Here $[\cdot ,\cdot ]$ and $\alpha$ in $\mathfrak{L}^{\ast}$ are dual to $\Delta$ and $\alpha$, respectively, in $\mathfrak{L}$.\\
$\bullet$ Conversely, if $(\mathfrak{L}, [\cdot ,\cdot ] , \alpha)$ is a finite dimensional Hom-Lie superalgebra, then $(\mathfrak{L}^{\ast}, \Delta , \alpha)$ is a Hom-Lie supercoalgebra, where $\Delta$ and $\alpha$ in $\mathfrak{L}^{\ast}$ are dual to $[\cdot ,\cdot ]$ and $\alpha$, respectively, in $\mathfrak{L}$.
\end{remark}
\begin{thm}\label{amineeeee} Let $(\mathfrak{L}, [\cdot ,\cdot ], \Delta, \alpha)$ be a finite dimensional (multiplicative) Hom-Lie superbialgebra. Then its linear dual $\mathfrak{L^\ast}=$Hom$(\mathfrak{L}, \mathbb{K})$ is also a (multiplicative) Hom-Lie superbialgebra with the dual structure maps:
\begin{equation}\label{amine}
\alpha(\phi)= \phi\circ\alpha,\ \ \
\langle[\phi, \psi], x \rangle=\langle \phi\otimes\psi, \Delta(x)\rangle, \ \ \
\langle\Delta(\phi), x\otimes y\rangle=\langle \phi, [x,y]\rangle,\\ \end{equation}
for $x,y\in\mathfrak{L}$ and $\phi,\psi\in\mathfrak{L^\ast}$.
\end{thm}
\begin{proof} As we mentioned right after Remark \ref{01}, $(\mathfrak{L}^{\ast}, [\cdot ,\cdot ] , \alpha)$ is a Hom-Lie superalgebra, which is true even if $\mathfrak{L}$ is not finite dimensional. Moreover, $(\mathfrak{L}^{\ast}, \Delta, \alpha)$ is a Hom-Lie supercoalgebra, whose validity depends on the finite dimensionality of $\mathfrak{L}$. Thus, it remains to check the compatibility condition (\ref{a}) between the bracket and the cobracket in $\mathfrak{L}^{\ast}$, i.e.,
\begin{eqnarray}\label{cc}
\langle\Delta([\phi, \psi]),x\otimes y\rangle=\langle ad_{\alpha(\phi)}(\Delta(\psi))-(-1)^{{|\phi|}{|\psi|}}ad_{\alpha(\psi)}(\Delta(\phi)), x\otimes y\rangle
\end{eqnarray}
for $x,y\in\mathfrak{L}$ and $\phi,\psi\in\mathfrak{L^\ast}$.\\
Using Definition \ref{amine}, the compatibility condition (\ref{a}) in $\mathfrak{L}$, we compute the left-hand side of (\ref{cc}) as follows:\\
\begin{align*}
\langle\Delta([\phi, \psi]),x\otimes y\rangle&=
\langle[\phi, \psi], [x, y]\rangle\\
&=\langle \phi\otimes\psi, \Delta([x, y])\rangle\\
&=\langle \phi\otimes\psi,   ad_{\alpha(x)}(\Delta(y))-(-1)^{{|x|}{|y|}}ad_{\alpha(y)}(\Delta(x))\rangle\\
&=\langle \phi\otimes\psi,[\alpha(x), y_{1}]\otimes\alpha(y_{2})\rangle + (-1)^{{|x|}{|y_{1}|}}\langle \phi\otimes\psi, \alpha(y_{1})\otimes[\alpha(x), y_{2}]\rangle\\
&-(-1)^{{|x|}{|y|}}\langle\phi\otimes\psi, [\alpha(y), x_{1}]\otimes\alpha(x_{2})\rangle - (-1)^{{|x|}{|y|}}(-1)^{{|y|}{|x_{1}|}}\langle \phi\otimes\psi, \alpha(x_{1})\otimes[\alpha(y), x_{2}]\rangle\\
&=\langle \Delta(\phi)\otimes\psi,\alpha(x)\otimes y_{1}\otimes\alpha(y_{2})\rangle + (-1)^{{|x|}{|y_{1}|}}\langle \phi\otimes\Delta(\psi), \alpha(y_{1})\otimes\alpha(x)\otimes y_{2}\rangle\\
&-(-1)^{{|x|}{|y|}}\langle\Delta(\phi)\otimes\psi, \alpha(y)\otimes x_{1}\otimes\alpha(x_{2})\rangle\\ &-(-1)^{{|x|}{|y|}}(-1)^{{|y|}{|x_{1}|}}\langle \phi\otimes\Delta (\psi), \alpha(x_{1})\otimes\alpha(y)\otimes x_{2}\rangle\\
&=\langle \phi_{1}\otimes\phi_{2}\otimes\psi,\alpha(x)\otimes y_{1}\otimes\alpha(y_{2})\rangle + (-1)^{{|x|}{|y_{1}|}}\langle \phi\otimes\psi_{1}\otimes\psi_{2}, \alpha(y_{1})\otimes\alpha(x)\otimes y_{2}\rangle\\
&-(-1)^{{|x|}{|y|}}\langle\phi_{1}\otimes\phi_{2}\otimes\psi, \alpha(y)\otimes x_{1}\otimes\alpha(x_{2})\rangle\\ &-(-1)^{{|x|}{|y|}}(-1)^{{|y|}{|x_{1}|}}\langle\phi\otimes\psi_{1}\otimes\psi_{2}, \alpha(x_{1})\otimes\alpha(y)\otimes x_{2}\rangle\\
&=\langle \alpha(\phi_{1})\otimes\phi_{2}\otimes\alpha(\psi), x\otimes y_{1}\otimes y_{2}\rangle + (-1)^{{|x|}{|y_{1}|}}\langle \alpha(\phi)\otimes\alpha(\psi_{1})\otimes\psi_{2}, y_{1}\otimes x\otimes y_{2}\rangle \\
&-(-1)^{{|x|}{|y|}}\langle\alpha(\phi_{1})\otimes\phi_{2}\otimes\alpha(\psi), y\otimes x_{1}\otimes x_{2}\rangle\\ &-(-1)^{{|x|}{|y|}}(-1)^{{|y|}{|x_{1}|}}\langle\alpha(\phi)\otimes\alpha(\psi_{1})\otimes\psi_{2}, x_{1}\otimes y\otimes x_{2}\rangle.
\end{align*}
Using, in addition, the skew-supersymmetric of the bracket and the cobracket in $\mathfrak{L}^{\ast}$,\\
 $([\phi_{1}, \alpha(\psi)]=-(-1)^{{|\phi_{1}|}{|\psi|}}[\alpha(\psi), \phi_{1}]$ and $\phi_{1}\otimes\phi_{2}=-(-1)^{{|\phi_{1}|}{|\phi_{2}|}}\phi_{2}\otimes\phi_{1}$).\\
The above four terms become:
\begin{align*}
&=-(-1)^{{|\phi|}{|\psi|}}(-1)^{{|\psi|}{|\phi_{1}|}}\langle\alpha(\phi_{1})\otimes[\alpha(\psi), \phi_{2}], x\otimes y\rangle + (-1)^{{|\phi|}{|\psi_{1}|}}\langle\alpha(\psi_{1})\otimes[\alpha(\phi), \psi_{2}], x\otimes y\rangle\\
&-(-1)^{{|\phi|}{|\psi|}}\langle[\alpha(\psi), \phi_{1}]\otimes\alpha(\phi_{2}), x\otimes y\rangle + \langle[\alpha(\phi), \psi_{1}]\otimes\alpha(\psi_{2}), x\otimes y\rangle.\\
&=\langle ad_{\alpha(\phi)}(\Delta(\psi))-(-1)^{{|\phi|}{|\psi|}}ad_{\alpha(\psi)}(\Delta(\phi)), x\otimes y\rangle
\end{align*}
This is exactly the right-hand side of (\ref{cc}).
\end{proof}
\section{Matched pairs, Hom-Lie superbialgebras and Manin supertriples}
In this section, we introduce the notions of  matched pair of Hom-Lie superalgebras and a Manin supertriple of Hom-Lie superalgebras. First, we recall the basics about   representations of a Hom-Lie superalgebras.

\begin{definition}\label{samak} (\cite{Ammar F Makhlouf A 2,Sheng,Sheng1}) Let $(\mathfrak{L}, [\cdot ,\cdot ], \alpha)$ be a Hom-Lie superalgebra and $M=M_{\bar{0}}\oplus M_{\bar{1}}$ an arbitrary vector superspace. A  representation  of the  Hom-Lie superalgebra with respect to $A\in \mathfrak{gl}(M)$, is an even linear map $\rho:\mathfrak{L} \rightarrow End(M)$, such that $\rho(\mathfrak{L_{i}})(M_{j})\subset M_{i+j}$ where $i, j\in\mathbb{Z}_{2}$, and satisfying
\begin{equation}\label{1 condition}
\rho(\alpha(x))\circ A=A\circ\rho(x)
\end{equation}
and
\begin{equation}\label{2 condition}
\rho([x, y])\circ A=\rho(\alpha(x))\circ\rho(y)-(-1)^{{|x|}{|y|}}\rho(\alpha(y))\circ\rho(x).
\end{equation}
for all homogeneous elements $x, y \in\mathfrak{L}$.\\
We denote a representation by $(M,  \rho,A)$. It is straightforward to see that $(\mathfrak{L}, ad, \alpha)$ is a representation, called the adjoint representation, see \cite{Benayadi Makhlouf,Sheng}.\\
Given a representation $(M,  \rho,A)$, define $\rho^{\ast}:\mathfrak{L}\rightarrow End(M^{\ast})$ by
\begin{equation}
\langle\rho^{\ast}(x)(\xi), \nu\rangle=-(-1)^{{|x|}{|\xi|}}\langle\xi,\rho(x)(\nu)\rangle,
\end{equation}
$\forall$ $x\in\mathfrak{L}$, $\xi\in M^{\ast}$, $\nu\in M.$ This representation is called admissible representation with respect to $(M,  \rho,A)$.
\end{definition}


Now, let $(\mathfrak{g}, [\cdot ,\cdot ]_{\mathfrak{g}}, \alpha_{\mathfrak{g}})$ and $(\mathfrak{g'}, [\cdot ,\cdot ]_{\mathfrak{g'}}, \alpha_{\mathfrak{g'}})$ be two multiplicative Hom-Lie superalgebras. Set $\mathfrak{g}=\mathfrak{g}_{\bar{0}}\oplus\mathfrak{g}_{\bar{1}}$ and $\mathfrak{g'}=\mathfrak{g'}_{\bar{0}}\oplus\mathfrak{g'}_{\bar{1}}$. Let $\rho:\mathfrak{g} \rightarrow \mathfrak{gl}(\mathfrak{g'})$ and $\rho':\mathfrak{g'} \rightarrow \mathfrak{gl}(\mathfrak{g})$ be two linear maps. 
Define a skew-supersymmetric bracket $[\cdot ,\cdot ]_{\widetilde{G}}: \widetilde{G}\times\widetilde{G}\rightarrow \widetilde{G}$, where $\widetilde{G}$ is given by : $\widetilde{G}=\widetilde{G}_{\bar{0}}\oplus \widetilde{G}_{\bar{1}}$, where $\widetilde{G}_{\bar{0}}=\mathfrak{g}_{\bar{0}}\oplus\mathfrak{g'}_{\bar{0}}$ and $\widetilde{G}_{\bar{1}}=\mathfrak{g}_{\bar{1}}\oplus\mathfrak{g'}_{\bar{1}}$. We set  $|(x,x')|=|x|=|x'|$ and $|(y,y')|=|y|=|y'|$, for all homogeneous elements $x, y$  in $\mathfrak{g}$ and $x', y'$ in  $\mathfrak{g'}$. Define $\widetilde{\alpha}:\widetilde{G}\rightarrow\widetilde{G}$ by
$$\widetilde{\alpha}(x, x')=(\alpha_{\mathfrak{g}}(x), \alpha_{\mathfrak{g'}}(x')),$$ and the bracket $\widetilde{G}$ by
\begin{equation}\label{zara}
\textbf{[}(x, x'), (y, y')\textbf{]}_{\widetilde{G}}=\textbf{(}[x, y]_{\mathfrak{g}}-(-1)^{{|x|}{|y|}}\rho'(y')(x)+\rho'(x')(y), [x', y']_{\mathfrak{g'}}+\rho(x)(y')-(-1)^{{|x|}{|y|}}\rho(y)(x')\textbf{)}.
\end{equation}
\begin{thm}\label{nouveau} The triple $(\widetilde{G}=\mathfrak{g}\oplus\mathfrak{g'}, [\cdot ,\cdot ]_{\widetilde{G}}, \widetilde{\alpha})$, where $\widetilde{G}$, $[\cdot ,\cdot ]_{\widetilde{G}}$, $\widetilde{\alpha}$ are defined above is a multiplicative Hom-Lie superalgebra if and only if $\rho$ and $\rho'$ are representations of $\mathfrak{g}$ and $\mathfrak{g'}$ respectively and the following conditions are satisfied
\begin{align}\label{seff}
\rho(\alpha_{\mathfrak{g}}(z))([x', y']_{\mathfrak{g'}})&=[\rho(z)(x'), \alpha_{\mathfrak{g'}}(y')]_{\mathfrak{g'}}+(-1)^{{|x|}{|z|}}[\alpha_{\mathfrak{g'}}(x'), \rho(z)(y')]_{\mathfrak{g'}}
\\&+(-1)^{{{|x|}{|y|}}+{{|y|}{|z|}}}\rho(\rho'(y')(z))(\alpha_{\mathfrak{g'}}(x'))-(-1)^{{|x|}{|z|}}\rho(\rho'(x')(z))(\alpha_{\mathfrak{g'}}(y')).\nonumber\end{align}
\begin{align}\label{chevro}
\rho'(\alpha_{\mathfrak{g'}}(z'))([x, y]_{\mathfrak{g}})&=[\rho'(z')(x), \alpha_{\mathfrak{g}}(y)]_{\mathfrak{g}}+(-1)^{{|x|}{|z|}}[\alpha_{\mathfrak{g}}(x), \rho'(z')(y)]_{\mathfrak{g}}\\
&+(-1)^{{{|x|}{|y|}}+{{|y|}{|z|}}}\rho'(\rho(y)(z'))(\alpha_{\mathfrak{g}}(x))-(-1)^{{|x|}{|z|}}\rho'(\rho(x)(z'))(\alpha_{\mathfrak{g}}(y)).\nonumber\end{align}
\end{thm}
\begin{proof} Assume $(\widetilde{G}=\mathfrak{g}\oplus\mathfrak{g'},  [\cdot ,\cdot ]_{\widetilde{G}}, \widetilde{\alpha})$ is a multiplicative Hom-Lie superalgebra. The multiplicativity condition writes 
\begin{equation}\label{multiplicative}
\widetilde{\alpha}\textbf{(}\textbf{[}(x, x'), (y, y')\textbf{]}_{\widetilde{G}}\textbf{)}=\textbf{[}\widetilde{\alpha}(x, x'), \widetilde{\alpha}(y, y')\textbf{]}_{\widetilde{G}}.
\end{equation}
Developing  (\ref{multiplicative}), leads to  the first condition  (\ref{1 condition}). Indeed 
\begin{align*}\widetilde{\alpha}\textbf{(}\textbf{[}(x, x'), (y, y')\textbf{]}_{\widetilde{G}}\textbf{)}&=\textbf{(}\alpha_{\mathfrak{g}}([x, y]_{\mathfrak{g}})-(-1)^{{|x|}{|y|}}\alpha_{\mathfrak{g}}(\rho'(y')(x))+\alpha_{\mathfrak{g}}(\rho'(x')(y)),\\& \alpha_{\mathfrak{g'}}([x', y']_{\mathfrak{g'}})+\alpha_{\mathfrak{g'}}(\rho(x)(y'))-(-1)^{{|x|}{|y|}}\alpha_{\mathfrak{g'}}(\rho(y)(x'))\textbf{)}.\end{align*}
\begin{align*}\textbf{[}\widetilde{\alpha}(x, x'), \widetilde{\alpha}(y, y')\textbf{]}_{\widetilde{G}}&
=\textbf{(}[\alpha_{\mathfrak{g}}(x), \alpha_{\mathfrak{g}}(y)]_{\mathfrak{g}}-(-1)^{{|x|}{|y|}}\rho'(\alpha_{\mathfrak{g'}}(y'))(\alpha_{\mathfrak{g}}(x))+
\rho'(\alpha_{\mathfrak{g'}}(x'))(\alpha_{\mathfrak{g}}(y)),\\& [\alpha_{\mathfrak{g'}}(x'), \alpha_{\mathfrak{g'}}(y')]_{\mathfrak{g'}}+\rho(\alpha_{\mathfrak{g}}(x))(\alpha_{\mathfrak{g'}}(y'))
-(-1)^{{|x|}{|y|}}\rho(\alpha_{\mathfrak{g}}(y))(\alpha_{\mathfrak{g'}}(x'))\textbf{)}.\end{align*}
Since $\alpha_{\mathfrak{g}}$ and $\alpha_{\mathfrak{g'}}$ are multiplicative, which implies that
\begin{equation}\label{57}
\rho(\alpha_{\mathfrak{g}}(x))\circ\alpha_{\mathfrak{g'}}=\alpha_{\mathfrak{g'}}\circ\rho(x),
\end{equation}
\begin{equation}\label{058}
\rho'(\alpha_{\mathfrak{g'}}(x'))\circ\alpha_{\mathfrak{g}}=\alpha_{\mathfrak{g}}\circ\rho'(x').
\end{equation}
Developing  the Hom-super-Jacobi identity (\ref{702}), that  for $x, y, z \in\mathfrak{g}$ and $x', y', z'\in\mathfrak{g'}$,\\
$$(-1)^{{|(x,x')|}{|(z, z')|}}[\widetilde{\alpha}(x, x'), [(y, y'), (z, z')]_{\widetilde{G}}]_{\widetilde{G}}+(-1)^{{|(y,y')|}{|(x, x')|}}[\widetilde{\alpha}(y, y'), [(z, z'), (x, x')]_{\widetilde{G}}]_{\widetilde{G}}$$
$$+(-1)^{{|(z,z')|}{|(y, y')|}}[\widetilde{\alpha}(z, z'), [(x, x'), (y, y')]_{\widetilde{G}}]_{\widetilde{G}}=0,$$
leads to  the second condition (\ref{2 condition}). Indeed,
\begin{align*}
\bullet &(-1)^{{|(x,x')|}{|(z, z')|}}[\widetilde{\alpha}(x, x'), [(y, y'), (z, z')]_{\widetilde{G}}]_{\widetilde{G}}=
(-1)^{{|x|}{|z|}}\textbf{(}[\alpha_{\mathfrak{g}}(x), [y, z]_{\mathfrak{g}}]_{\mathfrak{g}}-(-1)^{{|y|}{|z|}}[\alpha_{\mathfrak{g}}(x), \rho'(z')(y)]_{\mathfrak{g}}\\
&+[\alpha_{\mathfrak{g}}(x), \rho'(y')(z)]_{\mathfrak{g}}
-(-1)^{{|x|}({|y|}+{|z|})}\rho'([y', z']_{\mathfrak{g'}})(\alpha_{\mathfrak{g}}(x))-(-1)^{{|x|}({|y|}+{|z|})}\rho'(\rho(y)(z'))(\alpha_{\mathfrak{g}}(x))\\
&+(-1)^{{|x|}({|y|}+{|z|})}(-1)^{{|y|}{|z|}}\rho'(\rho(z)(y'))(\alpha_{\mathfrak{g}}(x))
+\rho'(\alpha_{\mathfrak{g'}}(x'))([y, z]_{\mathfrak{g}})\\
&-(-1)^{{|y|}{|z|}}\rho'(\alpha_{\mathfrak{g'}}(x'))(\rho'(z')(y))
+\rho'(\alpha_{\mathfrak{g'}}(x'))(\rho'(y')(z))\textbf{,}
\ \ [\alpha_{\mathfrak{g'}}(x'), [y', z']_{\mathfrak{g'}}]_{\mathfrak{g'}}\\
&+[\alpha_{\mathfrak{g'}}(x'),\rho(y)(z')]_{\mathfrak{g'}}
-(-1)^{{|y|}{|z|}}[\alpha_{\mathfrak{g'}}(x'), \rho(z)(y')]_{\mathfrak{g'}}
+\rho(\alpha_{\mathfrak{g}}(x))([y',z']_{\mathfrak{g'}})\\
&+\rho(\alpha_{\mathfrak{g}}(x))(\rho(y)(z'))
-(-1)^{{|y|}{|z|}}\rho(\alpha_{\mathfrak{g}}(x))(\rho(z)(y'))
-(-1)^{{|x|}({|y|}+{|z|})}\rho([y,z]_{\mathfrak{g}})(\alpha_{\mathfrak{g'}}(x'))\\
&+(-1)^{{|x|}({|y|}+{|z|})}(-1)^{{|y|}{|z|}}\rho(\rho'(z')(y))(\alpha_{\mathfrak{g'}}(x'))
-(-1)^{{|x|}({|y|}+{|z|})}\rho(\rho'(y')(z))(\alpha_{\mathfrak{g'}}(x'))\textbf{)}.\\
\end{align*}
Similarly, we compute  : $$(-1)^{{|(y,y')|}{|(x, x')|}}[\widetilde{\alpha}(y, y'), [(z, z'), (x, x')]_{\widetilde{G}}]_{\widetilde{G}} \ \  and \ \ (-1)^{{|(z,z')|}{|(y, y')|}}[\widetilde{\alpha}(z, z'), [(x, x'), (y, y')]_{\widetilde{G}}]_{\widetilde{G}}.$$
Setting  ($|x|=|x'|$,  $|y|=|y'|$, and $|z|=|z'|$ ) give
\begin{equation}\label{1122a}
\rho([x, y]_{\mathfrak{g}})\circ\alpha_{\mathfrak{g'}}=\rho(\alpha_{\mathfrak{g}}(x))\circ\rho(y)-(-1)^{{|x|}{|y|}}\rho(\alpha_{\mathfrak{g}}(y))\circ\rho(x),
\end{equation}
\begin{equation}\label{6060}
\rho'([x', y']_{\mathfrak{g'}})\circ\alpha_{\mathfrak{g}}=\rho'(\alpha_{\mathfrak{g'}}(x'))\circ\rho'(y')-(-1)^{{|x'|}{|y'|}}\rho'(\alpha_{\mathfrak{g'}}(y'))\circ\rho'(x'), \end{equation}
which implies  (\ref{seff}), and (\ref{chevro}).\\
By Eqs. (\ref{57}) and (\ref{1122a}), we deduce that $\rho$ is a representation of the Hom-Lie superalgebra $(\mathfrak{g}, [\cdot ,\cdot ]_{\mathfrak{g}}, \alpha_{\mathfrak{g}})$ on $\mathfrak{g'}$ with respect to $\alpha_{\mathfrak{g'}}$. By Eqs. (\ref{058}) and (\ref{6060}), we deduce that $\rho'$ is a representation of the Hom-Lie superalgebra $(\mathfrak{g'}, [\cdot ,\cdot ]_{\mathfrak{g'}}, \alpha_{\mathfrak{g'}})$ on $\mathfrak{g}$ with respect to $\alpha_{\mathfrak{g}}$.
\end{proof}
\begin{example}[See \cite{Ammar F Makhlouf A 2}] Given a representation $(V, [., .]_{V}, \beta)$ of a Hom-Lie superalgebra $(G, [\cdot ,\cdot ], \alpha)$. Set $\widetilde{G}=G\oplus V$ and $\widetilde{G_{k}}=G_{k}\oplus V_{k}$. If $x\in G_{i}$ and $v\in V_{i}$ $(i\in\mathbb{Z}_{2})$, we denote $|(x, v)|=|x|$.\\
Define a skew-supersymmetric bracket   $[\cdot ,\cdot ]_{\widetilde{G}}:\wedge^{2}(G\oplus V)\rightarrow G\oplus V$ by\\
$$[(x, \mu), (y, v)]_{\widetilde{G}}=([x, y], [x, v]_{V}-(-1)^{{|x|}{|y|}}[y, \mu]_{V}).$$\\
Define $\widetilde{\alpha}:G\oplus V\rightarrow G\oplus V$ by $\widetilde{\alpha}(x, v)=(\alpha(x), \beta(v))$. Then $(G\oplus V, [-, -]_{\widetilde{G}}, \widetilde{\alpha})$ is a Hom-Lie superalgebra, which we call the semi-direct product of the Hom-Lie superalgebra $(G, [\cdot ,\cdot ], \alpha)$ by $V$.
\end{example}
\begin{definition} For any $x\in\mathfrak{g}$, define $ad_{x}:\wedge^{n}\mathfrak{g}\rightarrow\wedge^{n}\mathfrak{g}$ by $ad_{x}P=[x, P]_{\mathfrak{g}}$. Its dual map $ad^{\ast}:\wedge^{n}\mathfrak{g}^{\ast}\rightarrow\wedge^{n}\mathfrak{g}^{\ast}$ is defined as 
\begin{equation}
\langle ad_{x}(y), \xi \rangle=-(-1)^{{|x|}{|y|}}\langle y, ad^{\ast}_{x}(\xi)\rangle,
\end{equation}
for all $x, y \in\mathfrak{g}$, and $\xi\in\mathfrak{g}^{\ast}$, where the pairing $\langle \   ,  \ \rangle$ is supersymmetric, i.e. $\langle x, \xi \rangle=(-1)^{{|x|}{|\xi|}}\langle \xi, x\rangle,$ for $x\in\mathfrak{g}$,\\ and $\xi\in\mathfrak{g}^{\ast}$. More precisely, for any $\xi_{1},\cdot\cdot\cdot,\xi_{n}\in\mathfrak{g}^{\ast}$, we have
\begin{equation}
ad^{\ast}_{x}(\xi_{1}\wedge\cdot\cdot\cdot\wedge\xi_{n})=\sum_{i=1}^{n} (-1)^{{|x|}({|\xi_{1}|+{|\xi_{2}|}}+\cdot\cdot\cdot+{|\xi_{i-1}|})}\alpha^{\ast}_{\mathfrak{g}}(\xi_{1})\wedge\cdot\cdot\cdot\wedge ad^{\ast}_{x}(\xi_{i})\wedge\cdot\cdot\cdot\wedge\alpha^{\ast}_{\mathfrak{g}}(\xi_{n}).
\end{equation}

\end{definition}

\begin{definition} A \emph{matched pairs of Hom-Lie superalgebras}, which we denote by $(\mathfrak{g}, \mathfrak{g'}, \rho, \rho')$, consists of two Hom-Lie superalgebras $(\mathfrak{g}, [\cdot ,\cdot ]_{\mathfrak{g}}, \alpha_{\mathfrak{g}})$ and $(\mathfrak{g'}, [\cdot ,\cdot ]_{\mathfrak{g'}}, \alpha_{\mathfrak{g'}})$, together with representations $\rho:\mathfrak{g} \rightarrow \mathfrak{gl}(\mathfrak{g'})$ and $\rho':\mathfrak{g'} \rightarrow \mathfrak{gl}(\mathfrak{g})$ with respect to $\alpha_{\mathfrak{g'}}$ and $\alpha_{\mathfrak{g}}$ respectively, such that the compatibility conditions (\ref{seff}) and (\ref{chevro}) are satisfied.
\end{definition}

In the following, we concentrate on the case that $\mathfrak{g'}$ is $\mathfrak{g}^{\ast}$, the dual space of $\mathfrak{g}$, and $ \alpha_{\mathfrak{g'}}= \alpha^{\ast}_{\mathfrak{g}}$, $\rho=ad^{\ast}$,  $\rho'=\mathfrak{ad}^{\ast}$, where $\mathfrak{ad}^{\ast}$ is the dual map of $\mathfrak{ad}$. Notice that $ad$ and $\mathfrak{ad}$ are  the adjoint representations associated to Hom-Lie superalgebras $\mathfrak{g}$ and $\mathfrak{g}^{\ast}$ respectively.  Let $x, y, z$ be  elements in $\mathfrak{g}$ and $\xi, \eta$  elements in $\mathfrak{g}^{\ast}$. 
\\
For a Hom-Lie superalgebra $(\mathfrak{g}, [\cdot ,\cdot ]_{\mathfrak{g}}, \alpha_{\mathfrak{g}})$ (resp. $(\mathfrak{g}^{\ast}, [\cdot ,\cdot ]_{\mathfrak{g}^{\ast}}, \alpha_{\mathfrak{g}^{\ast}})$), let $\Delta^{\ast}:\mathfrak{g}^{\ast}\rightarrow\wedge^{2}\mathfrak{g}^{\ast}$ (resp. $\Delta:\mathfrak{g}\rightarrow\wedge^{2}\mathfrak{g}$) be the dual map of $[\cdot ,\cdot ]_{\mathfrak{g}}:\wedge^{2}\mathfrak{g}\rightarrow\mathfrak{g}$ (resp. $[\cdot ,\cdot ]_{\mathfrak{g}^{\ast}}:\wedge^{2}\mathfrak{g}^{\ast}\rightarrow\mathfrak{g}^{\ast}$), i.e. $$\langle\Delta^{\ast}(\xi), x\wedge y \rangle =\langle \xi, [x, y]_{\mathfrak{g}}\rangle, \ \ \ \ \  \langle\Delta(x), \xi\wedge\eta \rangle =\langle x, [\xi, \eta]_{\mathfrak{g}^{\ast}}\rangle.$$

A Hom-Lie superalgebra $(\mathfrak{L}, [\cdot ,\cdot ], \alpha)$ is called \emph{admissible} if its adjoint representation is admissible, that is $[(Id-\alpha^{2})(x), \alpha(y)]=0$. In particular, if a Hom-Lie superalgebra $(\mathfrak{L}, [\cdot ,\cdot ], \alpha)$ is \emph{ involutive}, that is $\alpha^{2}=Id$, then it is admissible.

\begin{proposition}\label{ahmedd} A pair of admissible Hom-Lie superalgebras $(\mathfrak{g}, [\cdot ,\cdot ]_{\mathfrak{g}}, \alpha_{\mathfrak{g}})$ and $(\mathfrak{g}^{\ast}, [\cdot ,\cdot ]_{\mathfrak{g}^{\ast}}, \alpha_{\mathfrak{g}}^{\ast})$  determines a Hom-Lie superbialgebra $(\mathfrak{g}, [\cdot ,\cdot ]_{\mathfrak{g}},\Delta, \alpha_{\mathfrak{g}})$, where $\Delta$ is the dual operation of $ [\cdot ,\cdot ]_{\mathfrak{g}^{\ast}}$ if 
\begin{equation}\label{sure}
\langle\Delta([x, y]_{\mathfrak{g}}), \alpha_{\mathfrak{g}}^{\ast}(\xi)\wedge\eta \rangle = \langle ad_{\alpha_{\mathfrak{g}}(x)}(\Delta(y)), \alpha_{\mathfrak{g}}^{\ast}(\xi)\wedge\eta \rangle - (-1)^{{|x|}{|y|}} \langle ad_{\alpha_{\mathfrak{g}}(y)}(\Delta(x)), \alpha_{\mathfrak{g}}^{\ast}(\xi)\wedge\eta \rangle,
\end{equation}
\begin{equation}\label{2bia}
\langle\Delta^{\ast}([\xi, \eta]_{\mathfrak{g}^{\ast}}), \alpha_{\mathfrak{g}}(x)\wedge y \rangle = \langle \mathfrak{ad}_{{\alpha_{\mathfrak{g}}^{\ast}(\xi)}}(\Delta^{\ast}(\eta)), \alpha_{\mathfrak{g}}(x)\wedge y \rangle - (-1)^{{|\xi|}{|\eta|}} \langle \mathfrak{ad}_{{\alpha_{\mathfrak{g}}^{\ast}(\eta)}}(\Delta^{\ast}(\xi)), \alpha_{\mathfrak{g}}(x)\wedge y \rangle.
\end{equation}
\end{proposition}
\begin{remark} Following Bai and Sheng, we may denote this Hom-Lie superbialgebra by $(\mathfrak{g}, \mathfrak{g}^{\ast})$.
\end{remark}
\begin{thm}\label{bnbn10} Let  $(\mathfrak{g}, [\cdot ,\cdot ]_{\mathfrak{g}}, \alpha_{\mathfrak{g}})$ and $(\mathfrak{g}^{\ast}, [\cdot ,\cdot ]_{\mathfrak{g}^{\ast}}, \alpha_{\mathfrak{g}}^{\ast})$ be a pair of admissible Hom-Lie superalgebra. Then $(\mathfrak{g}, [\cdot ,\cdot ]_{\mathfrak{g}},\Delta, \alpha_{\mathfrak{g}})$, where $\Delta$ is the dual of $[\cdot ,\cdot ]_{\mathfrak{g}^{\ast}}$, is a Hom-Lie superbialgebra  if and only if $(\mathfrak{g}, [\cdot ,\cdot ]_{\mathfrak{g}}, \alpha_{\mathfrak{g}})$ and $(\mathfrak{g}^{\ast}, [\cdot ,\cdot ]_{\mathfrak{g}^{\ast}}, \alpha_{\mathfrak{g}}^{\ast})$ is a matched pairs of Hom-Lie superalgebras, i.e. $(\mathfrak{g}\oplus\mathfrak{g}^{\ast}, [\cdot ,\cdot ]_{\widetilde{G}}, \alpha_{\mathfrak{g}}\oplus\alpha_{\mathfrak{g}}^{\ast})$ is a multiplicative Hom-Lie superalgebra, where $[\cdot ,\cdot ]_{\widetilde{G}}$ is given by Eq. (\ref{zara}), in which $\rho=ad^{\ast}$ and $\rho'=\mathfrak{ad}^{\ast}.$
\end{thm}
\begin{proof} By Theorem \ref{nouveau}, two admissible Hom-Lie superalgebras $(\mathfrak{g}, [\cdot ,\cdot ]_{\mathfrak{g}}, \alpha_{\mathfrak{g}})$ and $(\mathfrak{g}^{\ast}, [\cdot ,\cdot ]_{\mathfrak{g}^{\ast}}, \alpha_{\mathfrak{g}}^{\ast})$ form a matched pair of Hom-Lie superalgebras if and only if
\begin{equation}\label{superrr}
ad^{\ast}_{\alpha_{\mathfrak{g}}(z)}([\xi, \eta]_{\mathfrak{g}^{\ast}})=[ad^{\ast}_{z}(\xi), \alpha_{\mathfrak{g}}^{\ast}(\eta)]_{\mathfrak{g}^{\ast}}+(-1)^{{|\xi|}{|z|}}[\alpha_{\mathfrak{g}}^{\ast}(\xi), ad^{\ast}_{z}(\eta)]_{\mathfrak{g}^{\ast}}
\end{equation}
\hspace{6cm} $+(-1)^{{{|\xi|}{|\eta|}}+{{|\eta|}{|z|}}}ad^{\ast}_{\mathfrak{ad}^{\ast}_{\eta}(z)}(\alpha_{\mathfrak{g}}^{\ast}(\xi))
-(-1)^{{|\xi|}{|z|}}ad^{\ast}_{\mathfrak{ad}^{\ast}_{\xi}(z)}(\alpha_{\mathfrak{g}}^{\ast}(\eta)).$\\
\begin{equation}\label{supeddddd}
\mathfrak{ad}^{\ast}_{\alpha_{\mathfrak{g}}^{\ast}(\xi)}([x,y]_{\mathfrak{g}})=[\mathfrak{ad}^{\ast}_{\xi}(x), \alpha_{\mathfrak{g}}(y)]_{\mathfrak{g}}+(-1)^{{|x|}{|\xi|}}[\alpha_{\mathfrak{g}}(x), \mathfrak{ad}^{\ast}_{\xi}(y)]_{\mathfrak{g}}
\end{equation}
\hspace{6cm} $+(-1)^{{{|x|}{|y|}}+{{|y|}{|\xi|}}}\mathfrak{ad}^{\ast}_{ad^{\ast}_{y}(\xi)}(\alpha_{\mathfrak{g}}(x))
-(-1)^{{|x|}{|\xi|}}\mathfrak{ad}^{\ast}_{ad^{\ast}_{x}(\xi)}(\alpha_{\mathfrak{g}}(y)).$\\
Since  $(\mathfrak{g}, [\cdot ,\cdot ]_{\mathfrak{g}}, \alpha_{\mathfrak{g}})$ and $(\mathfrak{g}^{\ast}, [\cdot ,\cdot ]_{\mathfrak{g}^{\ast}}, \alpha_{\mathfrak{g}}^{\ast})$ are admissible Hom-Lie superalgebras  i.e. ($\alpha_{\mathfrak{g}}^{2}=Id$, and $(\alpha_{\mathfrak{g}}^{\ast})^{2}=Id$),  see Lemma 2.9 in \cite{Sheng1}.
By Eq. (\ref{supeddddd}), we get
\begin{align*}
0&=\langle -\mathfrak{ad}^{\ast}_{\alpha_{\mathfrak{g}}^{\ast}(\xi)}([x,y]_{\mathfrak{g}})+[\mathfrak{ad}^{\ast}_{\xi}(x), \alpha_{\mathfrak{g}}(y)]_{\mathfrak{g}}+(-1)^{{|x|}{|\xi|}}[\alpha_{\mathfrak{g}}(x), \mathfrak{ad}^{\ast}_{\xi}(y)]_{\mathfrak{g}} +(-1)^{{{|x|}{|y|}}+{{|y|}{|\xi|}}}\mathfrak{ad}^{\ast}_{ad^{\ast}_{y}(\xi)}(\alpha_{\mathfrak{g}}(x))\\ &-(-1)^{{|x|}{|\xi|}}\mathfrak{ad}^{\ast}_{ad^{\ast}_{x}(\xi)}(\alpha_{\mathfrak{g}}(y))\textbf{,}\ \eta \rangle\\
&=(-1)^{{|\xi|}({|x|}+{|y|})}\langle [x,y]_{\mathfrak{g}}, [\alpha_{\mathfrak{g}}^{\ast}(\xi), \eta]_{\mathfrak{g}^{\ast}} \rangle-(-1)^{{|y|}({|x|}+{|\xi|})}\langle ad_{\alpha_{\mathfrak{g}}(y)}(\mathfrak{ad}^{\ast}_{\xi}(x)), \eta \rangle
+(-1)^{{|x|}{|\xi|}}\langle ad_{\alpha_{\mathfrak{g}}(x)}(\mathfrak{ad}^{\ast}_{\xi}(y)),\eta \rangle\\
&-(-1)^{{|\xi|}({|x|}+{|y|})}\langle\alpha_{\mathfrak{g}}(x), [ad^{\ast}_{y}(\xi), \eta]_{\mathfrak{g}^{\ast}} \rangle +(-1)^{{|x|}{|\xi|}}(-1)^{{|y|}({|x|}+{|\xi|})}\langle\alpha_{\mathfrak{g}}(y), [ad^{\ast}_{x}(\xi), \eta]_{\mathfrak{g}^{\ast}}\rangle\\
&=(-1)^{{|\xi|}({|x|}+{|y|})}\langle [x,y]_{\mathfrak{g}}, [\alpha_{\mathfrak{g}}^{\ast}(\xi), \eta]_{\mathfrak{g}^{\ast}} \rangle-(-1)^{{|x|}{|\xi|}}\langle x, [\xi, ad^{\ast}_{\alpha_{\mathfrak{g}}(y)}(\eta)]_{\mathfrak{g}^{\ast}}\rangle+(-1)^{{|y|}({|x|}+{|\xi|})}\langle y, [\xi, ad^{\ast}_{\alpha_{\mathfrak{g}}(x)}(\eta)]_{\mathfrak{g}^{\ast}}\rangle\\
&-(-1)^{{|\xi|}({|x|}+{|y|})}\langle x, [\alpha_{\mathfrak{g}}^{\ast}(ad^{\ast}_{y}(\xi)), \alpha_{\mathfrak{g}}^{\ast}(\eta)]_{\mathfrak{g}^{\ast}} \rangle +(-1)^{{|x|}{|\xi|}}(-1)^{{|y|}({|x|}+{|\xi|})}\langle y, [\alpha_{\mathfrak{g}}^{\ast}(ad^{\ast}_{x}(\xi)), \alpha_{\mathfrak{g}}^{\ast}(\eta)]_{\mathfrak{g}^{\ast}}\rangle,
\end{align*}
\begin{align*}
&=(-1)^{{|\xi|}({|x|}+{|y|})}\langle [x,y]_{\mathfrak{g}}, [\alpha_{\mathfrak{g}}^{\ast}(\xi), \eta]_{\mathfrak{g}^{\ast}} \rangle-(-1)^{{|x|}{|\xi|}}\langle x, [(\alpha_{\mathfrak{g}}^{\ast})^{2}(\xi), ad^{\ast}_{\alpha_{\mathfrak{g}}(y)}(\eta)]_{\mathfrak{g}^{\ast}}\rangle\\
&+(-1)^{{|y|}({|x|}+{|\xi|})}\langle y, [(\alpha_{\mathfrak{g}}^{\ast})^{2}(\xi), ad^{\ast}_{\alpha_{\mathfrak{g}}(x)}(\eta)]_{\mathfrak{g}^{\ast}}\rangle-(-1)^{{|\xi|}({|x|}+{|y|})}\langle x, [ad^{\ast}_{\alpha_{\mathfrak{g}}(y)}(\alpha_{\mathfrak{g}}^{\ast}(\xi)), \alpha_{\mathfrak{g}}^{\ast}(\eta)]_{\mathfrak{g}^{\ast}} \rangle\\ &+(-1)^{{|x|}{|\xi|}}(-1)^{{|y|}({|x|}+{|\xi|})}\langle y, [ad^{\ast}_{\alpha_{\mathfrak{g}}(x)}(\alpha_{\mathfrak{g}}^{\ast}(\xi)), \alpha_{\mathfrak{g}}^{\ast}(\eta)]_{\mathfrak{g}^{\ast}}\rangle\\
&=(-1)^{{|\xi|}({|x|}+{|y|})}\langle \Delta([x,y]_{\mathfrak{g}}), \alpha_{\mathfrak{g}}^{\ast}(\xi)\wedge\eta \rangle-(-1)^{{|x|}{|\xi|}}\langle \Delta(x), (\alpha_{\mathfrak{g}}^{\ast})^{2}(\xi)\wedge ad^{\ast}_{\alpha_{\mathfrak{g}}(y)}(\eta)\rangle\\
&+(-1)^{{|y|}({|x|}+{|\xi|})}\langle\Delta(y), (\alpha_{\mathfrak{g}}^{\ast})^{2}(\xi)\wedge ad^{\ast}_{\alpha_{\mathfrak{g}}(x)}(\eta)\rangle-(-1)^{{|\xi|}({|x|}+{|y|})}\langle \Delta(x), ad^{\ast}_{\alpha_{\mathfrak{g}}(y)}(\alpha_{\mathfrak{g}}^{\ast}(\xi))\wedge \alpha_{\mathfrak{g}}^{\ast}(\eta)\rangle\\ &+(-1)^{{|x|}{|\xi|}}(-1)^{{|y|}({|x|}+{|\xi|})}\langle\Delta(y), ad^{\ast}_{\alpha_{\mathfrak{g}}(x)}(\alpha_{\mathfrak{g}}^{\ast}(\xi))\wedge \alpha_{\mathfrak{g}}^{\ast}(\eta)\rangle,
\end{align*}
which implies that
\begin{align*}
\langle\Delta([x,y]_{\mathfrak{g}}), \alpha_{\mathfrak{g}}^{\ast}(\xi)\wedge\eta \rangle&=-(-1)^{{|x|}{|y|}}\langle \Delta(y), ad^{\ast}_{\alpha_{\mathfrak{g}}(x)}(\alpha_{\mathfrak{g}}^{\ast}(\xi))\wedge \alpha_{\mathfrak{g}}^{\ast}(\eta)+(-1)^{{|x|}{|\xi|}}(\alpha_{\mathfrak{g}}^{\ast})^{2}(\xi)\wedge ad^{\ast}_{\alpha_{\mathfrak{g}}(x)}(\eta)\rangle\\
&+\langle\Delta(x),ad^{\ast}_{\alpha_{\mathfrak{g}}(y)}(\alpha_{\mathfrak{g}}^{\ast}(\xi))\wedge \alpha_{\mathfrak{g}}^{\ast}(\eta)+(-1)^{{|y|}{|\xi|}}(\alpha_{\mathfrak{g}}^{\ast})^{2}(\xi)\wedge ad^{\ast}_{\alpha_{\mathfrak{g}}(y)}(\eta)\rangle\\
&=-(-1)^{{|x|}{|y|}}\langle\Delta(y), ad^{\ast}_{\alpha_{\mathfrak{g}}(x)}(\alpha_{\mathfrak{g}}^{\ast}(\xi)\wedge \eta)\rangle+\langle\Delta(x), ad^{\ast}_{\alpha_{\mathfrak{g}}(y)}(\alpha_{\mathfrak{g}}^{\ast}(\xi)\wedge \eta)\rangle\\
&=\langle ad_{\alpha_{\mathfrak{g}}(x)}(\Delta(y)), \alpha_{\mathfrak{g}}^{\ast}(\xi)\wedge\eta \rangle - (-1)^{{|x|}{|y|}} \langle ad_{\alpha_{\mathfrak{g}}(y)}(\Delta(x)),  \alpha_{\mathfrak{g}}^{\ast}(\xi)\wedge\eta\rangle ,
\end{align*}
which is exactly Eq. (\ref{sure}). Similarly, one deduces that Eq. (\ref{superrr}) is equivalent to Eq.  (\ref{2bia}).
\end{proof}

 Let $V$ be a superspace and $\langle\ , \ \rangle:V^{\ast}\times V\rightarrow \mathbb{K}$ be the canonical pairing. Then we identify $V$ with $V^{\ast}$ by the pairing $\langle x, \xi \rangle=(-1)^{{|x|}{|\xi|}}\langle\xi, x \rangle$, $x\in V$ and $\xi\in V^{\ast}$. On the other hand, we shall say that a bilinear form $(|):V\times V\rightarrow \mathbb{K}$ is supersymmetric if $(\upsilon|\omega)=(-1)^{{|\upsilon|}{|\omega|}}(\omega|\upsilon)$.

\begin{definition} A \emph{Manin supertriple of Hom-Lie superalgebras} is a triple of Hom-Lie superalgebras $(\mathfrak{M}, \mathfrak{g}, \mathfrak{g'})$ together with a nondegenerate supersymmetric  bilinear form $S$ on $\mathfrak{M}$ such that
\begin{enumerate}
\item  $S$ is invariant, i.e. for any $x, y, z \in\mathfrak{M}$, we have
\begin{equation}\label{manin 1}
S([x, y]_{\mathfrak{M}}, z)=S(x, [y, z]_{\mathfrak{M}}),
\end{equation}
\begin{equation}\label{manin 02}
S(\phi_{\mathfrak{M}}(x), y)=S(x, \phi_{\mathfrak{M}}(y)).
\end{equation}
\item $ \mathfrak{g}$ and $\mathfrak{g'}$ are isotropic Hom-Lie sub-superalgebra of $\mathfrak{M}$, such that $\mathfrak{M}=\mathfrak{g}\oplus\mathfrak{g'}$ as vector superspace.
\end{enumerate}
\end{definition}
\begin{proposition}\label{russe} Let $(\mathfrak{g}, \mathfrak{g}^{\ast})$ be a Hom-Lie superbialgebra in the sense of Proposition \ref{ahmedd}. Then $(\mathfrak{g}\oplus\mathfrak{g}^{\ast}, \mathfrak{g}, \mathfrak{g}^{\ast})$ is a Manin supertriple of Hom-Lie superalgebras.
\end{proposition}
\begin{proof} Let $(\mathfrak{g}, \mathfrak{g}^{\ast})$ be a Hom-Lie superbialgebra in the sense of Proposition \ref{ahmedd}, i.e. $\mathfrak{g}$ and $\mathfrak{g}^{\ast}$ are admissible Hom-Lie superalgebras such that Eqs. (\ref{sure}), (\ref{2bia}) are satisfied. By Theorem \ref{bnbn10}, we know that $(\mathfrak{g}\oplus\mathfrak{g}^{\ast}, [\cdot ,\cdot ]_{\mathfrak{g}\oplus\mathfrak{g}^{\ast}}, \alpha_{\mathfrak{g}}\oplus\alpha_{\mathfrak{g}}^{\ast})$ is a Hom-Lie superalgebra, where $[-, -]_{\mathfrak{g}\oplus\mathfrak{g}^{\ast}}$ is given by
\begin{equation}\label{omi}
[x+\xi, y+\eta]_{\mathfrak{g}\oplus\mathfrak{g}^{\ast}}=[x, y]_{\mathfrak{g}}+[\xi, \eta]_{\mathfrak{g}^{\ast}}+ad^{\ast}_{x}(\eta)-(-1)^{{|x|}{|y|}}ad^{\ast}_{y}(\xi)+
\mathfrak{ad}^{\ast}_{\xi}(y)-(-1)^{{|x|}{|y|}}\mathfrak{ad}^{\ast}_{\eta}(x),
\end{equation}
\ \ for all homogeneous elements $x, y$ and $\xi, \eta$ in $\mathfrak{g}$ and $\mathfrak{g}^{\ast}$ respectively.
From the construction above we have $|x|=|\xi|$ and $|y|=|\eta|$.\\
Furthermore, there is an obvious supersymmetric bilinear form on $\mathfrak{g}\oplus\mathfrak{g}^{\ast}$:\\
\begin{equation}\label{pose}
S(x+\xi, y+\eta)=\langle x, \eta\rangle+\langle \xi, y\rangle=\langle x, \eta\rangle+(-1)^{{|\xi|}{|y|}}\langle y, \xi\rangle.
\end{equation}
\ \ It's straightforward, using the   supersymmetry and $\langle ad_{x}(y), \xi \rangle=-(-1)^{{|x|}{|y|}}\langle y, ad^{\ast}_{x}(\xi)\rangle$, with $|x|=|\xi|$ and $|y|=|\eta|$, that  Eqs. (\ref{manin 1}) and (\ref{manin 02}) are satisfied, i.e. the bilinear form defined by Eq. (\ref{pose}) is invariant.
\end{proof}
Conversely, if $(\mathfrak{g}\oplus{\mathfrak{g}}^{\ast}, \mathfrak{g}, {\mathfrak{g}}^{\ast})$ is a Manin supertriple of Hom-Lie superalgebras with the invariant bilinear from $S$ given by Eq. (\ref{pose}), then for any $x, y\in\mathfrak{g}$ and $\xi, \eta\in{\mathfrak{g}}^{\ast}$, we have  the natural scalar product on $\mathfrak{g}\oplus{\mathfrak{g}}^{\ast}$  defined by $$S(x, y)=0, \ \ \ S(\xi, \eta)=0, \ \ \ S(x, \xi)=\langle x, \xi \rangle, \ \ \ \ x, y\in\mathfrak{g}, \  \xi, \eta\in{\mathfrak{g}}^{\ast}.$$ Due to the invariance of $S$, we have
\begin{align*}
S([x, \xi]_{\mathfrak{g}\oplus{\mathfrak{g}}^{\ast}}, y)&=(-1)^{{|y|}({|x|}+{|\xi|})}S(y,[x, \xi]_{\mathfrak{g}\oplus{\mathfrak{g}}^{\ast}})=(-1)^{{|y|}({|x|}+{|\xi|})}S([y, x]_{\mathfrak{g}\oplus{\mathfrak{g}}^{\ast}}, \xi)=(-1)^{{|y|}({|x|}+{|\xi|})}S([y, x]_{\mathfrak{g}},\xi)\\
&=-(-1)^{{|y|}{|\xi|}}\langle ad_{x}(y), \xi\rangle=(-1)^{{|y|}({|x|}+{|\xi|})}\langle y, ad^{\ast}_{x}(\xi)\rangle=\langle ad^{\ast}_{x}(\xi), y\rangle,
\end{align*}
$S([x, \xi]_{\mathfrak{g}\oplus{\mathfrak{g}}^{\ast}}, \eta)=S(x, [\xi, \eta]_{\mathfrak{g}\oplus{\mathfrak{g}}^{\ast}})=S(x, [\xi, \eta]_{{\mathfrak{g}}^{\ast}})=\langle x, \mathfrak{ad}_{\xi}(\eta)\rangle=-(-1)^{{|x|}{|\xi|}}\langle\mathfrak{ad}^{\ast}_{\xi}(x),
\eta\rangle,$\\
which implies that : $$[x, \xi]_{\mathfrak{g}\oplus{\mathfrak{g}}^{\ast}}=ad^{\ast}_{x}(\xi)-(-1)^{{|x|}{|\xi|}}\mathfrak{ad}^{\ast}_{\xi}(x),$$
that is, the Hom-Lie bracket on $\mathfrak{g}\oplus{\mathfrak{g}}^{\ast}$ is given by Eq. (\ref{omi}). Therefore, $(\mathfrak{g}, {\mathfrak{g}}^{\ast}, ad^{\ast}, \mathfrak{ad}^{\ast})$ is a matched pair of Hom-Lie superalgebras and hence $(\mathfrak{g}, {\mathfrak{g}}^{\ast})$ is a Hom-Lie superbialgebra. Note that we deduce naturally that both $\mathfrak{g}$ and $\mathfrak{g}^{\ast}$ are admissible  Hom-Lie superalgebras.\\
Summarizing the above study, Theorem \ref{bnbn10} and Proposition \ref{russe}, we have the following conclusion.

\begin{thm} Let $(\mathfrak{g}, [\cdot ,\cdot ]_{\mathfrak{g}}, \alpha_{\mathfrak{g}})$ and $(\mathfrak{g}^{\ast}, [\cdot ,\cdot ]_{\mathfrak{g}^{\ast}}, \alpha_{\mathfrak{g}}^{\ast})$ be two admissible Hom-Lie superalgebras. Then the following conditions are equivalent.
\begin{enumerate}
\item $(\mathfrak{g}, {\mathfrak{g}}^{\ast})$ is a Hom-Lie superbialgebra in the sense of Proposition \ref{ahmedd}.
\item $(\mathfrak{g}, {\mathfrak{g}}^{\ast}, ad^{\ast}, \mathfrak{ad}^{\ast})$ is a matched pair of Hom-Lie superalgebras.
\item $(\mathfrak{g}\oplus{\mathfrak{g}}^{\ast}, \mathfrak{g}, {\mathfrak{g}}^{\ast})$ is a Manin supertriple of Hom-Lie superalgebras with the invariant bilinear from (\ref{pose}).
\end{enumerate}

\end{thm}

\section{Coboundary and quasi-triangular Hom-Lie superbialgebras}
In this section, we define and study coboundary Hom-Lie superbialgebras and quasi-triangular Hom-Lie superbialgebras. Then we show how a coboundary or a quasi-triangular Hom-Lie superbialgebra can be constructed from a Hom-Lie superalgebra and an  $r$-matrix.
\begin{definition} A \emph{(multiplicative) coboundary Hom-Lie superbialgebra}  $(\mathfrak{L}, [\cdot ,\cdot ], \Delta, \alpha, r)$ consists of 
 a (multiplicative) Hom-Lie superbialgebra $(\mathfrak{L}, [\cdot ,\cdot ], \Delta, \alpha)$ and an element $r=\sum r_{1}\otimes r_{2}\in\mathfrak{L}^{\otimes2}$
such that $\alpha^{\otimes2}(r)=r$
and
\begin{equation}\label{888}
\Delta(x)=ad_{x}(r)=\sum [x, r_{1}]\otimes\alpha(r_{2})+(-1)^{{|x|}{|r_{1}|}}\alpha(r_{1})\otimes[x, r_{2}]
\end{equation}
for all $x\in\mathfrak{L}$.
\end{definition}
\begin{remark}
\begin{equation}\label{777}
\Delta(x)=(-1)^{{|x|}{|r|}}(\sum [x, r_{1}]\otimes\alpha(r_{2})+(-1)^{{|x|}{| r_{1}|}}\alpha(r_{1})\otimes[x, r_{2}])
\end{equation}
for $x\in\mathfrak{L}$, where the parity $|r|$ of $r$ is defined as follows : since we assume $r$ is homogenous, there exists $|r|\in\mathbb{Z}_{2}$, such that $r$ can be written as $r=\sum r_{1}\otimes r_{2}\in\mathfrak{L}^{\otimes2}$, $r_{1}, r_{2}$ are homogenous elements with $|r|=|r_{1}|+|r_{2}|$. (Note that equation (\ref{777}) and (\ref{2007}) show that we  have $|r|=\bar{0}$,  namely $|r_{1}|=|r_{2}|$). So we get (\ref{888}).

\end{remark}
\begin{definition}
The classical Yang-Baxter equation (CYBE):\\
$$c(r)=[r_{12}, r_{13}]+[r_{12}, r_{23}]+[r_{13}, r_{23}]=0,$$
where $r_{ij}$ are defined by\begin{eqnarray}\label{10}
  &&
  r_{12}=\sum r_{1}\otimes r_{2}\otimes1=r\otimes1,
   \\ \nonumber
  &&
 r_{13}=\sum r_{1}\otimes1 \otimes r_{2}=(1\otimes\tau)(r\otimes1)=(\tau\otimes1)(1\otimes r),\\ \nonumber
  &&
  r_{23}=\sum 1\otimes r_{1}\otimes r_{2}=1\otimes r
 \end{eqnarray}



and considered as elements in  ${U}(\mathfrak{L})$, the universal enveloping algebra of a Lie superalgebra $\mathfrak{L}$. 
 Elements \eqref{10} belongs to $\mathfrak{L}\otimes\mathfrak{L}\otimes\mathfrak{L}$.\\
 \end{definition}
 \begin{definition}
 The classical Hom-Yang-Baxter equation (CHYBE) in a Hom-Lie superalgebra $(\mathfrak{L}, [\cdot ,\cdot ], \alpha)$ is
 \begin{equation}\label{1000}
[[r,r]]^{\alpha}=[r_{12}, r_{13}]+[r_{12}, r_{23}]+[r_{13}, r_{23}]=0
\end{equation}
for $r\in\mathfrak{L}\otimes\mathfrak{L}$. The three brackets in (\ref{1000}) are defined as \\
\begin{eqnarray}\label{cvc}
  &&
   [r_{12},r'_{13}]=\sum (-1)^{{| r'_{1}|}{|r_{2}|}}[r_{1}, r'_{1}]\otimes \alpha(r_{2})\otimes \alpha(r'_{2}),
   \\ \nonumber
  &&
 [r_{12},r'_{23}]=\sum \alpha(r_{1})\otimes [r_{2},r'_{1}]\otimes \alpha(r'_{2}),   \\ \nonumber
  &&
  [r_{13},r'_{23}]=\sum (-1)^{{|r'_{1}|}{| r_{2}|}} \alpha(r_{1}) \otimes \alpha(r'_{1}) \otimes [r_{2}, r'_{2}],
 \end{eqnarray}
 where $r=\sum r_{1}\otimes r_{2}$ and $r'=\sum r'_{1}\otimes r'_{2}$ $\in\mathfrak{L}\otimes\mathfrak{L}$. If $\alpha=Id$, then the (CHYBE) reduces to the (CYBE).
\end{definition}

\begin{definition} A \emph{(multiplicative) quasi-triangular Hom-Lie superbialgebra} is a (multiplicative) coboundary Hom-Lie superbialgebra in which $r$ is a solution of the CHYBE (\ref{1000}). In these cases, we also write $\Delta$ as $ad(r)$.
\end{definition}

\begin{remark}Note that we do not require $r$ to be skew-supersymmetric in a coboundary Hom-Lie superbialgebra, whereas in \cite{Drinfel'd V.G.4} $r$ is assumed to be skew-supersymmetric in a coboundary Lie bialgebra. We follow the  convention in \cite{Majid} and \cite{Yau2}.
\end{remark}
\begin{remark}\label{masr} Condition \eqref{888} is a natural because from Remark \ref{2000}  the compatibility condition (\ref{a}) in Hom-Lie superbialgebra $\mathfrak{L}$ says that the cobracket $\Delta$ is a 1-cocycle in $ C^1(\mathfrak{L},\mathfrak{L}^{\otimes2})$, where $\mathfrak{L}$ acts on $\mathfrak{L}^{\otimes2}$ via the $\alpha$-twisted adjoint action (\ref{2001}). The simplest 1-cocycles are the 1-coboundaries, i.e, images of $\delta^{0}_{HL}$. We can define the Hom-Lie 0-cochains and 0 th differential as follows, extending the definition in \cite{MakhloufSilvestrov3}. Set $C^0(\mathfrak{L},\mathfrak{L}^{\otimes2})$ as the subspace of $\mathfrak{L}^{\otimes2}$ consisting of elements that are fixed by $\alpha^{\otimes2}$. Then we define the differential
$$\delta^{0}_{HL}:C^0(\mathfrak{L},\mathfrak{L}^{\otimes2})\rightarrow C^1(\mathfrak{L},\mathfrak{L}^{\otimes2})$$
by setting
$\delta^{0}_{HL}(r)=ad(r),$
as in (\ref{action}). It is not hard to check that, for $r\in C^0(\mathfrak{L},\mathfrak{L}^{\otimes2})$, we have $\delta^{1}_{HL}(\delta^{0}_{HL}(r))=0,$
where $\delta^{1}_{HL}$ is defined in (\ref{fff}). In fact, what this condition says is that
\begin{eqnarray}\label{rrrr}
  &&
   0=\delta^{1}_{HL}(\delta^{0}_{HL}(r))(x, y)
   \\ \nonumber
  &&
  \ \ = \delta^{1}_{HL}(ad(r))(x, y)  
    =ad_{[x, y]}(r)-ad_{\alpha(x)}(ad_{y}(r))+(-1)^{{|x|}{|y|}}ad_{\alpha(y)}(ad_{x}(r))
 \end{eqnarray}
for all $x, y\in\mathfrak{L}$. We will prove (\ref{rrrr}) in Lemma \ref{compatiblite} below. Thus, such an element $\delta^{0}_{HL}(r)=ad(r)$ is a 1-coboundary, and hence a 1-cocycle. This fact makes $ad(r)$ (with $\alpha^{\otimes2}(r)=r$) a natural candidate for the cobracket in a Hom-Lie superbialgebra and also justifies the name coboundary Hom-Lie superbialgebra.
\end{remark}
The following result is the analogue of Theorem \ref{aaa} for coboundary or quasi-triangular Hom-Lie superbialgebras. It says that coboundary or quasi-triangular Hom-Lie superbialgebras deform into other coboundary or quasi-triangular Hom-Lie superbialgebras via suitable endomorphisms.
\begin{thm}\label{zz} Let $(\mathfrak{L}, [\cdot ,\cdot ], \Delta=ad(r), \alpha, r)$ be a coboundary Hom-Lie superbialgebra and an even map $\beta:\mathfrak{L}\rightarrow \mathfrak{L}$ be a morphism such that $\beta^{\otimes2}(r)=r$. Then
$\mathfrak{L}_{\beta}=(\mathfrak{L},  [\cdot ,\cdot ]_{\beta}=\beta\circ[\cdot ,\cdot ], \Delta_{\beta}=\Delta\circ\beta,\beta\alpha, r)$
is also a coboundary Hom-Lie superbialgebra, which is multiplicative if $\mathfrak{L}$ is. Moreover, if $\mathfrak{L}$ is quasi-triangular, then so is $\mathfrak{L}_{\beta}$.
\end{thm}
\begin{proof} By Theorem \ref{aaa} we know that $\mathfrak{L}_{\beta}$ is a Hom-Lie superbialgebra, multiplicative if $\mathfrak{L}$ is. To check that $\mathfrak{L}_{\beta}$ is coboundary, first note that
$(\beta\alpha)^{\otimes2}(r)=\beta^{\otimes2}\alpha^{\otimes2}(r) = r.$

To check the condition (\ref{888}) in $\mathfrak{L}_{\beta}$, we compute as follows:
\begin{align*}
\Delta_{\beta}(x)&=\beta^{\otimes2}(\Delta(x))
=\beta^{\otimes2}([x, r_{1}]\otimes\alpha(r_{2}))+\beta^{\otimes2}((-1)^{{|x|}{| r_{1}|}}\alpha(r_{1})\otimes[x, r_{2}])\\
&=[x, r_{1}]_{\beta}\otimes\beta\alpha(r_{2})+(-1)^{{|x|}{| r_{1}|}}\beta\alpha(r_{1})\otimes[x, r_{2}]_{\beta}.
\end{align*}
The last expression above is $ad_{x}(r)$ in $\mathfrak{L}_{\beta}$, which shows that $\mathfrak{L}_{\beta}$ is coboundary.\\
Finally, suppose in addition that $\mathfrak{L}$ is quasi-triangular, i.e., $r$ is a solution of the CHYBE in  $\mathfrak{L}$. Using the notation in (\ref{1000}) we have:
\begin{align*}
0=\beta^{\otimes3}([r_{12}, r_{13}]+[r_{12}, r_{23}]+[r_{13}, r_{23}])
=[r_{12}, r_{13}]_{\beta}+[r_{12}, r_{23}]_{\beta}+[r_{13}, r_{23}]_{\beta},
\end{align*}
where the last expression is defined in $\mathfrak{L}_{\beta}$. This shows that $r$ is a solution of the CHYBE in $\mathfrak{L}_{\beta}$, so $\mathfrak{L}_{\beta}$ is quasi-triangular.\\
The following result is the analogue of Corollary \ref{aaaa} for coboundary or quasi-triangular Hom-Lie superbialgebras. It says that these objects can be obtained by twisting coboundary or quasi-triangular Lie superbialgebras via suitable endomorphisms.
\end{proof}
\begin{cor} Let $(\mathfrak{L}, [\cdot ,\cdot ], \Delta, r)$ be a coboundary Lie superbialgebra and an even map $\beta:\mathfrak{L}\rightarrow \mathfrak{L}$ be a Lie superalgebra morphism such that $\beta^{\otimes2}(r)=r$. Then
$\mathfrak{L}_{\beta}=(\mathfrak{L},  [\cdot ,\cdot ]_{\beta}=\beta\circ[\cdot ,\cdot ], \Delta_{\beta}=\Delta\circ\beta,\beta, r)$
is a multiplicative coboundary Hom-Lie superbialgebra. If, in addition, $\mathfrak{L}$ is a quasi-triangular Lie superbialgebra, then $\mathfrak{L}_{\beta}$ is a multiplicative quasi-triangular Hom-Lie superbialgebra.
\end{cor}
\begin{proof} This is the $\alpha=Id$ special case of Theorem \ref{zz}, provided that we can show that\\ $\Delta\circ\beta=\beta^{\otimes2}\circ\Delta$. We compute as follows:
\begin{align*}
\beta^{\otimes2}(\Delta(x))&=\beta^{\otimes2}(ad_{x}(r))
=\beta^{\otimes2}([x, r_{1}]\otimes r_{2})+\beta^{\otimes2}((-1)^{{|x|}{|r_{1}|}}r_{1}\otimes[x, r_{2}])\\
&=[\beta(x), \beta(r_{1})]\otimes\beta(r_{2})+(-1)^{{|x|}{|r_{1}|}}\beta(r_{1})\otimes[\beta(x), \beta(r_{2})]\\
&=[\beta(x), r_{1}]\otimes r_{2}+(-1)^{{|x|}{|r_{1}|}}r_{1}\otimes[\beta(x), r_{2}]
=ad_{\beta(x)}(r)
=\Delta(\beta(x)).
\end{align*}
\end{proof}
The next result says that every multiplicative coboundary or quasi-triangular Hom-Lie superbialgebra gives rise to an infinite sequence of multiplicative coboundary or quasi-triangular Hom-Lie superbialgebra. It is similar to Corollary \ref{www}.
\begin{cor} Let $(\mathfrak{L}, [\cdot ,\cdot ], \Delta, \alpha, r)$ be a multiplicative coboundary (resp. quasi-triangular) Hom-Lie superbialgebra. Then
$\mathfrak{L}_{\alpha^{n}}=(\mathfrak{L},  [\cdot ,\cdot ]_{\alpha^{n}}=\alpha^{n}\circ[\cdot ,\cdot ], \Delta_{\alpha^{n}}=\Delta\circ\alpha^{n},\alpha^{n+1}, r)$
is also a multiplicative coboundary (resp. quasi-triangular)  Hom-Lie superbialgebra for each integer $n\geq0$.
\end{cor}
\begin{proof} This is the $\beta=\alpha^{n}$ special case of Theorem \ref{zz}.
\end{proof}
In the following result, we describe some sufficient condition under which a Hom-Lie superalgebra becomes a coboundary Hom-Lie superbialgebra.\\
In what follows, for an element $r=\sum r_{1}\otimes r_{2}$, we write $r_{21}$ for $\tau(r)=\sum (-1)^{{|r_{1}|}{|r_{2}|}}r_{2}\otimes r_{1}.$
\begin{lem}\label{compatiblite} Let  $(\mathfrak{L}, [\cdot ,\cdot ], \alpha)$ be a multiplicative Hom-Lie superalgebra and $r\in\mathfrak{L}^{\otimes2}$ be an element such that $\alpha^{\otimes2}(r)=r$, (\ref{2007}) and (\ref{777}),  ((\ref{777})\ and\ (\ref{2007}),\  i.e., $|r|=\bar{0})$. Then $\Delta=ad(r):\mathfrak{L}\rightarrow\mathfrak{L}^{\otimes2}$ satisfies (\ref{a}), i.e.,
$ad_{[x, y]}(r)=ad_{\alpha(x)}(ad_{y}(r))-(-1)^{{|x|}{|y|}}ad_{\alpha(y)}(ad_{x}(r))$
for $x, y\in\mathfrak{L}$.
\end{lem}
\begin{proof} We will use $\alpha^{\otimes2}(r)=r$, the skew-supersymmetry (\ref{701}) and the Hom-Jacobi identity of (\ref{702}) and $\alpha([x,y])=[\alpha(x),\alpha(y)]$,  (multiplicative) in the computation below. For $x, y \in\mathfrak{L}$, we have:
\begin{align*}
ad_{[x, y]}(r)&=[[x, y], r_{1}]\otimes\alpha(r_{2})+(-1)^{{|[x, y]|}{|r_{1}|}}\alpha(r_{1})\otimes[[x, y], r_{2}]\\
&=[[x, y], \alpha(r_{1})]\otimes\alpha^{2}(r_{2})+(-1)^{{|x|}{|r_{1}|}}(-1)^{{|y|}{|r_{1}|}}\alpha^{2}(r_{1})\otimes[[x, y], \alpha(r_{2})]\\
&=\textbf{(}[\alpha(x), [y,r_{1}]]+(-1)^{{|x|}{|y|}}(-1)^{{|x|}{|r_{1}|}}[\alpha(y), [r_{1}, x]]\textbf{)}\otimes\alpha^{2}(r_{2})\\
&+(-1)^{{|x|}{|r_{1}|}}(-1)^{{|y|}{|r_{1}|}}\alpha^{2}(r_{1})\otimes\textbf{(}[\alpha(x), [y,r_{2}]]+(-1)^{{|x|}{|y|}}(-1)^{{|x|}{|r_{2}|}}[\alpha(y), [r_{2}, x]]\textbf{)}\\
&=[\alpha(x), [y,r_{1}]]\otimes\alpha^{2}(r_{2})+(-1)^{{|x|}{|y|}}(-1)^{{|x|}{|r_{1}|}}[\alpha(y), [r_{1}, x]]\otimes\alpha^{2}(r_{2})\\
&+(-1)^{{|x|}{|r_{1}|}}(-1)^{{|y|}{|r_{1}|}}\alpha^{2}(r_{1})\otimes[\alpha(x), [y,r_{2}]]+(-1)^{{|x|}{|y|}}(-1)^{{|y|}{|r_{1}|}}\alpha^{2}(r_{1})\otimes[\alpha(y), [r_{2}, x]]\\
&=[\alpha(x), [y,r_{1}]]\otimes\alpha^{2}(r_{2})+(-1)^{{|x|}{|y|}}(-1)^{{|x|}{|r_{1}|}}\alpha([y,r_{1}])\otimes
[\alpha(x), \alpha(r_{2})]\\
&+(-1)^{{|y|}{|r_{1}|}}[\alpha(x), \alpha(r_{1})]\otimes\alpha([y,r_{2}])+(-1)^{{|x|}{|r_{1}|}}(-1)^{{|y|}{|r_{1}|}}\alpha^{2}(r_{1})\otimes[\alpha(x), [y,r_{2}]]\\
&-(-1)^{{|x|}{|y|}}[\alpha(y), [x,r_{1}]]\otimes\alpha^{2}(r_{2})-(-1)^{{|y|}{|r_{1}|}}\alpha([x, r_{1}])\otimes[\alpha(y), \alpha(r_{2})]\\
&-(-1)^{{|x|}{|y|}}(-1)^{{|x|}{|r_{1}|}}[\alpha(y), \alpha(r_{1})]\otimes\alpha([x, r_{2}])-(-1)^{{|x|}{|y|}}(-1)^{{|y|}{|r_{1}|}+{|x|}{|r_{1}|}}\alpha^{2}(r_{1})\otimes[\alpha(y), [x,r_{2}]]\\
&=ad_{\alpha(x)}\textbf{(}[y,r_{1}]\otimes\alpha(r_{2})+(-1)^{{|y|}{|r_{1}|}}\alpha(r_{1})\otimes[y, r_{2}]\textbf{)}\\
&-(-1)^{{|x|}{|y|}}ad_{\alpha(y)}\textbf{(}[x,r_{1}]\otimes\alpha(r_{2})+(-1)^{{|x|}{|r_{1}|}}\alpha(r_{1})\otimes[x, r_{2}]\textbf{)}\\
&=ad_{\alpha(x)}(ad_{y}(r))-(-1)^{{|x|}{|y|}}ad_{\alpha(y)}(ad_{x}(r)).
\end{align*}
\end{proof}
\begin{thm}\label{Lie} Let $(\mathfrak{L}, [\cdot ,\cdot ], \alpha)$ be a multiplicative Hom-Lie superalgebra and $r\in\mathfrak{L}^{\otimes2}$ be an element such that $\alpha^{\otimes2}(r)=r, \ \ r_{21}=-r,$\  (\ref{2007}) and (\ref{777}),  ((\ref{777})\ and\ (\ref{2007}),\  i.e., $|r|=\bar{0})$,
and
\begin{equation}\label{ppp}
\alpha^{\otimes3}(ad_{x}([[r, r]]^{\alpha}))=0
\end{equation}
for all $x\in\mathfrak{L}$, where $[[r, r]]^{\alpha}$ is defined in (\ref{1000}). Define $\Delta:\mathfrak{L}\rightarrow\mathfrak{L}^{\otimes2}$ as $\Delta(x)=ad_{x}(r)$ as in (\ref{888}).\\
Then $(\mathfrak{L}, [\cdot ,\cdot ], \Delta, \alpha, r)$ is a multiplicative coboundary Hom-Lie superbialgebra.
\end{thm}

\begin{proof} We will show the following statements:\\
1) $\Delta=ad(r)$ commutes with $\alpha$.\\
2) $\Delta$ is skew-supersymmetric.\\
3) The compatibility condition (\ref{a}) holds.\\
4) The condition (\ref{ppp}) is equivalent Hom-super-coJacobi identity of $\Delta$ (\ref{jacobi}).\\
Write $r=\sum r_{1}\otimes r_{2}$, $r'=\sum r'_{1}\otimes r'_{2}$ and $\tau(r)=\sum (-1)^{{|r_{1}|}{|r_{2}|}}r_{2}\otimes r_{1}$, $\tau(r')=\sum (-1)^{{|r'_{1}|}{|r'_{2}|}}r'_{2}\otimes r'_{1}$. To show that $\Delta=ad(r)$ commutes with $\alpha$, pick an element $x\in\mathfrak{L}$, the summation sign will often be omitted in computation to simplify the typography. Using the definition $\Delta=ad(r)$, $\alpha([x, r_{1}])=[\alpha(x), \alpha(r_{1})]$, $\alpha([x, r_{2}])=[\alpha(x), \alpha(r_{2})]$ and the assumption $\alpha^{\otimes2}(r)=r$ we have
\begin{align*}
\Delta(\alpha(x))&=[\alpha(x), r_{1}]\otimes\alpha(r_{2})+(-1)^{{|x|}{|r_{1}|}}\alpha(r_{1})\otimes[\alpha(x), r_{2}]\\
&=\alpha([x, r_{1}])\otimes\alpha^{2}(r_{2})+(-1)^{{|x|}{|r_{1}|}}\alpha^{2}(r_{1})\otimes\alpha([x, r_{2}])
=\alpha^{\otimes2}(\Delta(x)).
\end{align*}
This shows that $\Delta$ commutes with $\alpha$.\\
Now we show that $\Delta=ad(r)$ is skew-supersymmetric. We have
\begin{align*}
\Delta(x)&=[x, r_{1}]\otimes\alpha(r_{2})+(-1)^{{|x|}{|r_{1}|}}\alpha(r_{1})\otimes[x, r_{2}].\\
\tau(\Delta(x))&=(-1)^{{|r_{1}|}{|r_{2}|}}([x, r_{2}]\otimes\alpha(r_{1})+(-1)^{{|x|}{|r_{2}|}}\alpha(r_{2})\otimes[x, r_{1}]).
\end{align*}
Then  $\Delta(x)+\tau(\Delta(x))=ad_{x}(r+r_{21})=ad_{x}(0)=0,$
since
\begin{align*}
r+r_{21}&=\sum(r_{1}\otimes r_{2}+(-1)^{{|r_{1}|}{|r_{2}|}}r_{2}\otimes r_{1}).
\end{align*}
We will prove that the compatibility condition (\ref{a}) holds in Lemma \ref{compatiblite}.\\
Finally, we show that the Hom-super-coJacobi identity (\ref{jacobi}) of $\Delta=ad(r)$ is equivalent to (\ref{ppp}). Let us unwrap the Hom-super-coJacobi identity. Fix an element $x\in\mathfrak{L}$, and let $r'=\sum r'_{1}\otimes r'_{2}$ be another copy of $r$. Then we write \begin{align*}
\omega&=(\alpha\otimes\Delta)(\Delta(x))\\
&=(\alpha\otimes\Delta)([x, r_{1}]\otimes\alpha(r_{2})+(-1)^{{|x|}{|r_{1}|}}\alpha(r_{1})\otimes[x, r_{2}])\\
&=\alpha([x, r_{1}])\otimes[\alpha( r_{2}), r'_{1}]\otimes\alpha( r'_{2})+(-1)^{{|r_{2}|}{|r'_{1}|}}\alpha([x, r_{1}])\otimes\alpha(r'_{1})\otimes[\alpha( r_{2}), r'_{2}]\\
&+(-1)^{{|x|}{|r_{1}|}}\alpha^2(r_{1})\otimes[[x, r_{2}], r'_{1}]\otimes\alpha(r'_{2})+(-1)^{{|x|}{|r_{1}|}}(-1)^{{|x|}{|r'_{1}|}}(-1)^{{|r_{2}|}{|r'_{1}|}}\alpha^2(r_{1})\otimes\alpha(r'_{1})\otimes[[x, r_{2}], r'_{2}].
\end{align*}
Note
\begin{align*}
A_{1}&=\alpha([x, r_{1}])\otimes[\alpha( r_{2}), r'_{1}]\otimes\alpha( r'_{2}),\\
B_{1}&=(-1)^{{|r_{2}|}{|r'_{1}|}}\alpha([x, r_{1}])\otimes\alpha(r'_{1})\otimes[\alpha( r_{2}), r'_{2}],\\
C_{1}&=(-1)^{{|x|}{|r_{1}|}}\alpha^2(r_{1})\otimes[[x, r_{2}], r'_{1}]\otimes\alpha(r'_{2}),\\
D_{1}&=(-1)^{{|x|}{|r_{1}|}}(-1)^{{|x|}{|r'_{1}|}}(-1)^{{|r_{2}|}{|r'_{1}|}}\alpha^2(r_{1})\otimes\alpha(r'_{1})\otimes[[x, r_{2}], r'_{2}].
\end{align*}
we get
$
\omega=A_{1}+B_{1}+C_{1}+D_{1}.
$

With these notations, the Hom-super-coJacobi identity of $\Delta=ad(r)$ (applied to $x$) becomes
\begin{equation}\label{anis3}
   (1\otimes1\otimes1+\xi+\xi^2)\circ(\alpha\otimes\Delta)\circ\Delta(x)
    = \sum_{i=1}^{3}(A_{i}+B_{i}+C_{i}+D_{i})=0.
 \end{equation}
Therefore, to prove the equivalence between the Hom-super-coJacobi identity of $\Delta$ and (\ref{ppp}), it suffices to show
\begin{equation}\label{Hom}
\alpha^{\otimes3}(ad_{x}([[r, r]]^{\alpha}))=\sum_{i=1}^{3}(A_{i}+B_{i}+C_{i}+D_{i}),
\end{equation}
which we will prove in Lemma \ref{xyz0} below.\\
The proof of Theorem \ref{Lie} will be completed once we prove the Lemma below.
\end{proof}
\begin{lem}\label{xyz0} The condition (\ref{Hom}) holds.
\end{lem}
\begin{proof} It suffices to show the following three equalities:\\
\begin{eqnarray}\label{102030}
  &&
   \alpha^{\otimes3}(ad_{x}([r_{12}, r_{13}]))=A_{3}+B_{2}+C_{3}+D_{2},
   \end{eqnarray}
\begin{eqnarray}\label{anis1}
  &&
   \alpha^{\otimes3}(ad_{x}([r_{12}, r_{23}]))=A_{1}+B_{3}+C_{1}+D_{3},
   \end{eqnarray}
\begin{eqnarray}\label{anis2}
  &&
   \alpha^{\otimes3}(ad_{x}([r_{13}, r_{23}]))=A_{2}+B_{1}+C_{2}+D_{1},
   \end{eqnarray}
where the three bracket, which add up to $[[r,r]]^{\alpha}$, are defined in (\ref{cvc}). The proofs for the three equalities are very similar, so we will only give the proof of (\ref{102030}).\\
Since $r=r'$ and $r_{12}=-r$, we have
\begin{eqnarray}
  &&
   A_{3}=(-1)^{{|x|}{|r_{2}|}}(-1)^{{|r'_{1}|}{|r_{1}|}}[\alpha(r_{2}), r'_{1}]\otimes\alpha(r'_{2})\otimes\alpha([x, r_{1}])
   \\ \nonumber
  &&
   \ \ \ \ =(-1)^{{|x|}{|r_{2}|}}(-1)^{{|r_{1}|}{|r'_{1}|}}[\alpha(r'_{2}), r_{1}]\otimes\alpha(r_{2})\otimes\alpha([x, r'_{1}])  \\ \nonumber
  &&
  \ \ \ \ =-(-1)^{{|x|}{|r_{2}|}}[\alpha(r'_{1}), r_{1}]\otimes\alpha(r_{2})\otimes\alpha([x, r'_{2}])\\ \nonumber
  &&
  \ \ \ \ =-(-1)^{{|x|}{|r_{2}|}}[\alpha^2(r'_{1}), \alpha^2(r_{1})]\otimes\alpha^3(r_{2})\otimes\alpha([x, \alpha(r'_{2})])\\ \nonumber
  &&
  \ \ \ \ =\alpha^{\otimes3}\textbf{(}(-1)^{{|x|}{|r_{2}|}}(-1)^{{|r'_{1}|}{|r_{2}|}}\alpha([r_{1}, r'_{1}])\otimes\alpha^2(r_{2})\otimes[x, \alpha(r'_{2})]\textbf{)}.
 \end{eqnarray}
In the fourth equality we used $\alpha^{\otimes2}(r)=r$, $\alpha^{\otimes2}(r')=r'$. In the equality we used the skew-supersymmetry of $[-, -]$ and $\alpha([r_{1}, r'_{1}])=[\alpha(r_{1}), \alpha(r'_{1})]$, we know $(-1)^{{|x|}{|r_{1}|}}(-1)^{{|x|}{|r'_{1}|}}=1$, just to check the calculations. Similar computation give
\begin{eqnarray*}
  &&
   B_{2}=(-1)^{{|r'_{1}|}{|r_{1}|}}[\alpha(r_{2}), r'_{2}]\otimes\alpha([x, r_{1}])\otimes\alpha(r'_{1})
  \\ \nonumber
&&=\alpha^{\otimes3}\textbf{(}(-1)^{{|x|}{|r_{1}|}}(-1)^{{|x|}{|r'_{1}|}}(-1)^{{|r'_{1}|}{|r_{2}|}}\alpha([r_{1}, r'_{1}])\otimes[x, \alpha(r_{2})]\otimes\alpha^2(r'_{2})\textbf{)},
   \end{eqnarray*}
   \begin{eqnarray*}
   C_{3}=(-1)^{{|r'_{1}|}{|r_{2}|}}[[x, r_{2}], r'_{1}]\otimes\alpha(r'_{2})\otimes\alpha^2(r_{1})
=(-1)^{{|x|}{|r'_{2}|}}[[r'_{2}, x], \alpha(r_{2})]\otimes\alpha^2(r_{1})\otimes\alpha^2(r'_{1}),
   \end{eqnarray*}
\begin{eqnarray*}
   D_{2}=(-1)^{{|r'_{1}|}{|r_{2}|}}[[x, r_{2}], r'_{2}]\otimes\alpha^2(r'_{2})\otimes\alpha(r'_{1})
=(-1)^{{|r'_{1}|}{|r_{2}|}}[[x, r_{2}], \alpha(r'_{2})]\otimes\alpha^2(r_{1})\otimes\alpha^2(r'_{1}). 
   \end{eqnarray*}
Using, in addition, the skew-supersymmetry (\ref{701}) and the Hom-Jacobi identity (\ref{702}) of $[-, -]$, we add $C_{3}$ and $D_{2}$:
\begin{align*}
C_{3}+D_{2}&=(-1)^{{|r'_{1}|}{|r_{2}|}}\textbf{(}(-1)^{{|x|}{|r'_{2}|}}(-1)^{{|r_{2}|}{|r'_{2}|}}[[r'_{2}, x], \alpha(r_{2})]+[[x, r_{2}], \alpha(r'_{2})]\textbf{)}\otimes\alpha^2(r_{1})\otimes\alpha^2(r'_{1})\\
&=(-1)^{{|r'_{1}|}{|r_{2}|}}[\alpha(x),[r_{2}, r'_{2}]]\otimes\alpha^2(r_{1})\otimes\alpha^2(r'_{1})\\
&=(-1)^{{|r'_{1}|}{|r_{2}|}}[\alpha(x),[r_{1}, r'_{1}]]\otimes\alpha^2(r_{2})\otimes\alpha^2(r'_{2})\\
&=(-1)^{{|r'_{1}|}{|r_{2}|}}[\alpha(x),[\alpha(r_{1}), \alpha(r'_{1})]]\otimes\alpha^3(r_{2})\otimes\alpha^3(r'_{2})\\
&=\alpha^{\otimes3}\textbf{(}(-1)^{{|r'_{1}|}{|r_{2}|}}[x, [r_{1}, r'_{1}]]\otimes\alpha^2(r_{2})\otimes\alpha^2(r'_{2})\textbf{)}.
\end{align*}
Using the definition (\ref{action}) of $ad_{x}$, we now conclude that:
\begin{align*}
A_{3}+B_{2}+C_{3}+D_{2}&=\alpha^{\otimes3}\textbf{(}(-1)^{{|x|}{|r_{2}|}}(-1)^{{|r'_{1}|}{|r_{2}|}}\alpha([r_{1}, r'_{1}])\otimes\alpha^2(r_{2})\otimes[x, \alpha(r'_{2})]\textbf{)}\\
&+\alpha^{\otimes3}\textbf{(}(-1)^{{|x|}{|r_{1}|}}(-1)^{{|x|}{|r'_{1}|}}(-1)^{{|r'_{1}|}{|r_{2}|}}\alpha([r_{1}, r'_{1}])\otimes[x, \alpha(r_{2})]\otimes\alpha^2(r'_{2})\textbf{)}\\
&+\alpha^{\otimes3}\textbf{(}(-1)^{{|r'_{1}|}{|r_{2}|}}[x, [r_{1}, r'_{1}]]\otimes\alpha^2(r_{2})\otimes\alpha^2(r'_{2})\textbf{)}\\
&=\alpha^{\otimes3}\textbf{(}(-1)^{{|r'_{1}|}{|r_{2}|}}ad_{x}([r_{1}, r'_{1}]\otimes\alpha(r_{2})\otimes\alpha(r'_{2}))\textbf{)}\\
&=\alpha^{\otimes3}\textbf{(}ad_{x}((-1)^{{|r'_{1}|}{|r_{2}|}}[r_{1}, r'_{1}]\otimes\alpha(r_{2})\otimes\alpha(r'_{2}))\textbf{)}\\
&=\alpha^{\otimes3}\textbf{(}ad_{x}([r_{12},r_{13}])\textbf{)}.
\end{align*}
This proves (\ref{102030}).\\
The equalities (\ref{anis1}) and (\ref{anis2}) are proved by very similar computations.\\
Therefore, the equality (\ref{Hom}) holds. Together with (\ref{anis3}) we have shown that the the Hom-super-coJacobi identity of $\Delta=ad(r)$ is equivalent to $\alpha^{\otimes3}(ad_{x}([[r, r]]^{\alpha}))=0$
\end{proof}
The following result is an immediate consequence of Theorem (\ref{Lie}). It gives sufficient conditions under which a Hom-Lie superalgebra becomes a quasi-triangular Hom-Lie superbialgebra.
\begin{cor} Let $(\mathfrak{L}, [\cdot ,\cdot ], \alpha)$ be a multiplicative Hom-Lie superalgebra and $r\in\mathfrak{L}^{\otimes2}$ be an element such that $\alpha^{\otimes2}(r)=r, \ \ r_{21}=-r,$\  (\ref{2007}) and (\ref{777}), ((\ref{777})\ and\ (\ref{2007}),\  i.e., $|r|=\bar{0})$, and $[[r, r]]^{\alpha}=0.$
Then $(\mathfrak{L}, [\cdot ,\cdot ], ad(r), \alpha, r)$ is a multiplicative quasi-triangular Hom-Lie superbialgebra.
\end{cor}
\begin{thm} Let $(\mathfrak{L}, [\cdot ,\cdot ], \Delta, \alpha, r)$ be a coboundary Hom-Lie superbialgebra. Then the following statements are equivalent, \\
(1) $\mathfrak{L}$ is a quasi-triangular Hom-Lie superbialgebra, i.e., $[[r,r]]^{\alpha}=0$ (\ref{1000}).\\
(2) The equality $(\alpha\otimes\Delta)(r)=-[r_{12}, r_{13}]$
holds, where the bracket is defined in (\ref{cvc}).\\
(3) The equality $(\Delta\otimes\alpha)(r)=[r_{13}, r_{23}]$
holds, where the bracket is defined in (\ref{cvc}).
\end{thm}
\begin{proof} The equivalence between three statements clearly follows from the equalities.
Let $r'=\sum r'_{1}\otimes r'_{2}$ be another copy of $r$. Since $\Delta=ad(r)$ (\ref{888}), $r=r'$ and $r_{1}\otimes r_{2}=-(-1)^{{|r_{1}|}{|r_{2}|}}r_{2}\otimes r_{1}$.
By calculation we will find results
\begin{align*}
(\alpha\otimes\Delta)(r_{1}\otimes r_{2})&=\alpha(r_{1})\otimes\Delta(r_{2})
=\alpha(r_{1})\otimes[r_{2}, r'_{1}]\otimes\alpha(r'_{2})+(-1)^{{|r'_{1}|}{|r_{2}|}}\alpha(r_{1})\otimes\alpha(r'_{1})\otimes[r_{2}, r'_{2}]\\
&=[r_{12}, r_{23}]+[r_{13}, r_{23}].
\end{align*}
This shows the equivalence between statements (1) and (2). Likewise, we have
\begin{align*}
(\Delta\otimes\alpha)(r_{1}\otimes r_{2})
&=\Delta(r_{1})\otimes\alpha(r_{2})
=[r_{1},r'_{1}]\otimes\alpha(r'_{2})\otimes\alpha(r_{2})+(-1)^{{|r_{1}|}{|r'_{1}|}}\alpha(r'_{1})\otimes[r_{1}, r'_{2}]\otimes\alpha(r_{2})\\
&=-(-1)^{{|r'_{1}|}{|r_{2}|}}[r_{1}, r'_{1}]\otimes\alpha(r_{2})\otimes\alpha(r'_{2})-\alpha(r_{1})\otimes[r_{2}, r'_{1}]\otimes\alpha(r'_{2})\\
&=-[r_{12}, r_{13}]-[r_{12}, r_{23}].
\end{align*}
This shows the equivalence between statements (1) and (3).
\end{proof}

\section{Cobracket perturbation  in Hom-Lie superbialgebras }
The purpose of this section is to study perturbation of cobrackets in Hom-Lie superbialgebras, following Drinfel'd's perturbation theory of quasi-Hopf algebras (\cite{Drinfel'd V.G.1}, \cite{Drinfel'd V.G.5}, \cite{Drinfel'd V.G.6}, \cite{Drinfel'd V.G.7}).\\
We address the following question :
" If $(\mathfrak{L}, [\cdot ,\cdot ], \Delta, \alpha)$ is a Hom-Lie superbialgebra (Definition \ref{ae}) and $t\in\mathfrak{L}^{\otimes2}$, under what conditions does the perturbed cobracket $\Delta_{t}=\Delta+ad(t)$
give another Hom-Lie superbialgebra $(\mathfrak{L}, [\cdot ,\cdot ],\Delta_{t},\alpha)$ ?"\\
\\
Define the perturbed cobracket $\Delta_{t}=\Delta+ad(t)$. For $x\in\mathfrak{L}$ and $t=\sum t_{1}\otimes t_{2}\in\mathfrak{L}^{\otimes2}$ and also recall the adjoint map $ad_{x}:\mathfrak{L}^{\otimes n}\rightarrow \mathfrak{L}^{\otimes n}$ (\ref{action}) we have :
\begin{align*}
\Delta_{t}(x)&=\Delta(x)+ad_{x}(t)
=\Delta(x)+ad_{x}(t_{1}\otimes t_{2})
=\Delta(x)+[x, t_{1}]\otimes\alpha(t_{2})+(-1)^{{|x|}{|t_{1}|}}\alpha(t_{1})\otimes[x, t_{2}].
\end{align*}
This is a natural question because $\Delta$ is a 1-cocycle (Remark \ref{2000}), $ad(t)$ (\ref{action}) is a 1-coboundary when $\alpha^{\otimes2}(t)=t$ (Remark \ref{masr}), and perturbation of cocycles by coboundaries is a natural concept in homological algebra.
Of course, we have more to worry about than just the cocycle condition (\ref{a}) because $(\mathfrak{L}, \Delta_{t},\alpha)$ must be a Hom-Lie supercoalgebra (Definition \ref{00001}).\\
\\
In the following result, we give some sufficient conditions under which the perturbed cobracket $\Delta_{t}$ gives another Hom-Lie superbialgebra. This is a generalization of \cite{Majid}, which deals with cobracket perturbation in Lie superbialgebras.\\
A result about cobracket perturbation in a quasi-triangular Hom-Lie superbialgebra, is given after the following result. We also briefly discuss triangular Hom-Lie superbialgebra, which is the Hom-Type version of Drinfel'd's triangular Lie bialgebra \cite{Drinfel'd V.G.4}.\\
Let us recall some notations first.
For $t=\sum t_{1}\otimes t_{2}\in\mathfrak{L}^{\otimes2}$, the symbol $t_{21}$ denotes $\tau(t)=\sum (-1)^{{|t_{1}|}{|t_{2}|}} t_{2}\otimes t_{1}$. If $\varphi(x, y)$ is an expression in the elements $x$ and $y$, we set
$$|\varphi(x, y)|=\varphi(x, y)-(-1)^{{|x|}{|y|}}\varphi(y, x).$$
For example, the compatibility condition
$\Delta([x,y])=ad_{\alpha(x)}(\Delta(y))-(-1)^{{|x|}{|y|}}ad_{\alpha(y)}(\Delta(x))$ (\ref{a})\\ is equivalent to $$\Delta([x,y])=|ad_{\alpha(x)}(\Delta(y))|.$$
Moreover $[[x, y], \alpha(z)]=|[\alpha(x), [y, z]]|$, where $|[\alpha(x), [y, z]]|=[\alpha(x), [y, z]]-(-1)^{{|x|}{|y|}}[\alpha(y), [x, z]]$.\\
By calculation the Hom-super-Jacobi identity (\ref{702}) is equivalent to $[[x, y], \alpha(z)]=|[\alpha(x), [y, z]]|$.\\
Note that we have $$|\varphi(x, y)+\psi(x, y)|=|\varphi(x, y)|+|\psi(x, y)|.$$
It is also noted to simplify writing $(1\otimes1\otimes1+\xi+\xi^2)=\circlearrowleft$.
\begin{thm}\label{monsieur} Let $(\mathfrak{L}, [\cdot ,\cdot ], \Delta, \alpha)$ be a multiplicative Hom-Lie superbialgebra and $t\in\mathfrak{L}^{\otimes2}$ be an element such that $\alpha^{\otimes2}(t)=t, \ \ t_{21}=-t,$\  (\ref{2007}) and (\ref{777}), ((\ref{777})\ and\ (\ref{2007}),\  i.e., $|t|=\bar{0})$
and
\begin{equation}\label{faxx}
\alpha^{\otimes3}(ad_{x}([[t, t]]^{\alpha}+\circlearrowleft(\alpha\otimes\Delta)(t))=0
\end{equation}
for all $x\in\mathfrak{L}$.
Then $\mathfrak{L}_{t}=(\mathfrak{L}, [\cdot ,\cdot ], \Delta_{t}=\Delta+ad(t),  \alpha)$ is multiplicative Hom-Lie superbialgebra.
\end{thm}
\begin{proof} To show that $\mathfrak{L}_{t}$ is a multiplicative Hom-Lie superbialgebra, we need to prove the following conditions:\\
It is clear that $(\mathfrak{L}, [\cdot ,\cdot ], \alpha)$ is a multiplicative Hom-Lie superalgebra.\\
It remains to show that $(\mathfrak{L}, \Delta_{t}, \alpha)$ co-multiplicative Hom-Lie supercoalgebra, and the compatibility condition (\ref{a}) holds for $\Delta_{t}$ and $[\cdot ,\cdot ]$.\\
Precisely we need to prove four things:\\
$\bullet$ $\alpha^{\otimes2}\circ\Delta_{t}=\Delta_{t}\circ\alpha$,  equality is true because:
\begin{align*}
\Delta_{t}(\alpha(x))&=\Delta(\alpha(x))+ad_{\alpha(x)}(t).\\
\alpha^{\otimes2}(\Delta_{t}(x))&=\alpha^{\otimes2}(\Delta(x))+\alpha^{\otimes2}(ad_{x}(t))
=\Delta(\alpha(x))+\alpha^{\otimes2}(ad_{x}(t_{1}\otimes t_{2})).
\end{align*}
Using $\alpha^{\otimes2}(t)=t$, we have
\begin{align*}
\alpha^{\otimes2}(ad_{x}(t_{1}\otimes t_{2}))&=[\alpha(x), t_{1}]\otimes\alpha(t_{2})+(-1)^{{|x|}{|t_{1}|}}\alpha(t_{1})\otimes[\alpha(x), t_{2}]
=ad_{\alpha(x)}(t).
\end{align*}
$\bullet$ $\Delta_{t}$ is skew-supersymmetric because:
\begin{align*}
\Delta_{t}(x)&=\Delta(x)+ad_{x}(t),\\
\tau(\Delta_{t}(x))&=-\Delta(x)+(-1)^{{|t_{1}|}{|t_{2}|}}\textbf{(}[x,t_{2}]\otimes\alpha(t_{1})+(-1)^{{|x|}{|t_{2}|}}\alpha(t_{2})\otimes[x,t_{1}]\textbf{)}.
\end{align*}
Then : $\Delta_{t}(x)+\tau(\Delta_{t}(x))=ad_{x}\textbf{(}t+(-1)^{{|t_{1}|}{|t_{2}|}}t_{2}\otimes t_{1}\textbf{)}=ad_{x}(t+t_{21})=ad_{x}(0)=0.$\\
$\bullet$ Now, we need to show the compatibility condition (\ref{a}) holds for $\Delta_{t}$ and $[\cdot ,\cdot ]$:
\begin{equation}\label{a2016}
\Delta_{t}([x,y])=ad_{\alpha(x)}(\Delta_{t}(y))-(-1)^{{|x|}{|y|}}ad_{\alpha(y)}(\Delta_{t}(x)),
\end{equation}
which  is equivalent to
\begin{equation}\label{add1}
\Delta_{t}([x,y])=|ad_{\alpha(x)}(\Delta_{t}(y))|.
\end{equation}
Since $\Delta_{t}=\Delta+ad(t)$,  
(\ref{add1}) is equivalent to
\begin{align*}
\Delta([x, y])+ad_{[x, y]}(t)&=|ad_{\alpha(x)}(\Delta(y))+ad_{\alpha(x)}(ad_{y}(t))|
=|ad_{\alpha(x)}(\Delta(y))|+|ad_{\alpha(x)}(ad_{y}(t))|.
\end{align*}
Moreover, since $\Delta([x,y])=ad_{\alpha(x)}(\Delta(y))-(-1)^{{|x|}{|y|}}ad_{\alpha(y)}(\Delta(x))=
|ad_{\alpha(x)}(\Delta(y))|$, because $\mathfrak{L}$ is a Hom-Lie superbialgebra, (\ref{add1}) is equivalent to,
$
ad_{[x, y]}(t)=|ad_{\alpha(x)}(ad_{y}(t))|,
$
which holds by Lemma \ref{compatiblite}.\\
$\bullet$ Finally, we must show the Hom-super-coJacobi identity of $\Delta_{t}$, which states
\begin{equation}\label{shss}
\circlearrowleft(\alpha\otimes\Delta_{t})(\Delta_{t}(x))=0
\end{equation}
for all $x\in\mathfrak{L}$. Using the definition $\Delta_{t}=\Delta+ad(t)$. We can rewrite (\ref{shss}) as
\begin{equation}\label{147a}
\circlearrowleft(\alpha\otimes\Delta)(\Delta(x))+\circlearrowleft(\alpha\otimes\Delta)(ad_{x}(t))+\\
\circlearrowleft(\alpha\otimes ad(t))(\Delta(x))+\circlearrowleft(\alpha\otimes ad(t))(ad_{x}(t))=0.
\end{equation}
We already know that $\circlearrowleft(\alpha\otimes\Delta)(\Delta(x))=0$,  which is the Hom-super-coJacobi identity of $\Delta$.\\
Moreover, in (\ref{anis3}) and (\ref{Hom}) (in the proof of Theorem \ref{Lie} with $t$ instead of $r$), we already showed that
\begin{equation}\label{147b}
\circlearrowleft(\alpha\otimes ad(t))(ad_{x}(t))=\alpha^{\otimes3}(ad_{x}([[t, t]]^{\alpha})).
\end{equation}
In view of (\ref{147a}) and (\ref{147b}), the Hom-super-coJacobi identity of $\Delta_{t}$ (\ref{shss}) is equivalent to
\begin{equation}\label{mosq}
\alpha^{\otimes3}(ad_{x}([[t, t]]^{\alpha}))+\circlearrowleft\textbf{(}(\alpha\otimes ad(t))(\Delta(x))+(\alpha\otimes\Delta)(ad_{x}(t))\textbf{)}=0.
\end{equation}
Using the assumption (\ref{faxx}), the condition (\ref{mosq}) is equivalent to
\begin{equation}\label{m10}
\circlearrowleft\textbf{(}(\alpha\otimes ad(t))(\Delta(x))+(\alpha\otimes\Delta)(ad_{x}(t))\textbf{)}=\alpha^{\otimes3}(\circlearrowleft ad_{x}((\alpha\otimes\Delta)(t))).
\end{equation}
We will prove (\ref{m10}) in Lemma \ref{edin} below.\\
The proof of Theorem \ref{monsieur} will be complete once we prove Lemma \ref{edin}.
\end{proof}
\begin{lem}\label{edin}
The condition (\ref{m10}) holds.
\end{lem}
\begin{proof} Write $\Delta(x)=\sum x_{1}\otimes x_{2}$ and $t=\sum t_{1}\otimes t_{2}\in\mathfrak{L}^{\otimes2}$. Then the left-hand side of (\ref{m10}) is:
\begin{align*}
&\circlearrowleft\textbf{(}(\alpha\otimes ad(t))(\Delta(x))+(\alpha\otimes\Delta)(ad_{x}(t))\textbf{)}\\
&= \ \circlearrowleft\textbf{(}\alpha(x_{1})\otimes ad_{x_{2}}(t_{1}\otimes t_{2})+(\alpha\otimes\Delta)([x, t_{1}]\otimes\alpha(t_{2})+(-1)^{{|x|}{|t_{1}|}}\alpha(t_{1})\otimes[x, t_{2}])\textbf{)}\\
&= \ \circlearrowleft\textbf{(}\alpha(x_{1})\otimes[x_{2}, t_{1}]\otimes\alpha(t_{2})+(-1)^{{|x_{2}|}{|t_{1}|}}\alpha(x_{1})\otimes\alpha(t_{1})\otimes
[x_{2}, t_{2}]\textbf{)}\\
&+\circlearrowleft\textbf{(}\alpha([x, t_{1}])\otimes\Delta(\alpha(t_{2}))+(-1)^{{|x|}{|t_{1}|}}\alpha^{2}(t_{1})\otimes\Delta([x, t_{2}])\textbf{)}.
\end{align*}
Write $\Delta(t_{2})=\sum t'_{2}\otimes t''_{2}$. Recall from (\ref{a}) that:
$\Delta([x, t_{2}])=ad_{\alpha(x)}(\Delta(t_{2}))-(-1)^{{|x|}{|t_{2}|}}ad_{\alpha(t_{2})}(\Delta(x))$, because $\mathfrak{L}$ is a Hom-Lie superbialgebra. We can continue the above computation as follows :\\
\begin{align*}
&= \ \circlearrowleft\textbf{(}\alpha(x_{1})\otimes[x_{2}, t_{1}]\otimes\alpha(t_{2})+(-1)^{{|x_{2}|}{|t_{1}|}}\alpha(x_{1})\otimes\alpha(t_{1})\otimes
[x_{2}, t_{2}]+\alpha([x, t_{1}])\otimes\alpha^{\otimes2}(\Delta(t_{2}))\textbf{)}\\
&+\circlearrowleft\textbf{(}(-1)^{{|x|}{|t_{1}|}}\alpha^{2}(t_{1})\otimes[\alpha(x), t'_{2}]\otimes\alpha(t''_{2})+(-1)^{{|x|}{|t_{1}|}}(-1)^{{|x|}{|t'_{2}|}}\alpha^{2}(t_{1})\otimes\alpha(t'_{2})\otimes
[\alpha(x), t''_{2}]\textbf{)}\\
&-\circlearrowleft\textbf{(}(-1)^{{|x|}{|t_{1}|}}(-1)^{{|x|}{|t_{2}|}}\alpha^{2}(t_{1})\otimes[\alpha(t_{2}),x_{1}]\otimes\alpha(x_{2})\\
&-(-1)^{{|x|}{|t_{1}|}}(-1)^{{|x|}{|t_{2}|}}(-1)^{{|x_{1}|}{|t_{2}|}}\alpha^{2}(t_{1})\otimes\alpha(x_{1})\otimes[\alpha(t_{2}), x_{2}]\textbf{)}.
\end{align*}
It follows the skew-supersymmetry of $\Delta$ applied to $x$ (i.e, $\sum x_{1}\otimes x_{2}=-\sum (-1)^{{|x_{1}|}{|x_{2}|}} x_{2}\otimes x_{1}$), $|t|=\bar{0}$, $t_{21}=-t$, and $\alpha^{\otimes2}(t)=t$, We find that \\
$\bullet$ $\alpha^{2}(t_{1})\otimes[\alpha(t_{2}),x_{1}]\otimes\alpha(x_{2})
=(-1)^{{|x_{1}|}{|x_{2}|}}\alpha(x_{2})\otimes\alpha^{2}(t_{1})\otimes[\alpha(t_{2}), x_{1}].$
\\
$\bullet$ $(-1)^{{|x_{1}|}{|t_{2}|}}\alpha^{2}(t_{1})\otimes\alpha(x_{1})\otimes[\alpha(t_{2}), x_{2}]=(-1)^{{|t_{1}|}{|t_{2}|}}(-1)^{{|x_{2}|}{|t_{1}|}}\alpha(x_{1})\otimes[\alpha(t_{2}), x_{2}]\otimes\alpha^{2}(t_{1}).$
That the first two terms and the last two terms above cancel out. Using the commutation of $\alpha$ with $[\cdot ,\cdot ]$ and $\Delta$ and $\alpha^{\otimes2}(t)=t$, the above computation continues as follows:
\begin{align*}
&= \ \circlearrowleft\textbf{(}\alpha([x, t_{1}])\otimes\alpha^{\otimes2}(\Delta(t_{2}))+(-1)^{{|x|}{|t_{1}|}}\alpha^{2}(t_{1})\otimes[\alpha(x), t'_{2}]\otimes\alpha(t''_{2})\\
&+(-1)^{{|x|}({|t_{1}|+|t'_{2}|})}\alpha^{2}(t_{1})\otimes\alpha(t'_{2})\otimes
[\alpha(x), t''_{2}]\textbf{)}\\
&= \ \circlearrowleft\textbf{(}\alpha([x, \alpha(t_{1})])\otimes\alpha^{\otimes2}(\Delta(\alpha(t_{2})))+(-1)^{{|x|}{|t_{1}|}}\alpha^{3}(t_{1})\otimes[\alpha(x), \alpha(t'_{2})]\otimes\alpha^{2}(t''_{2})\\
&+(-1)^{{|x|}({|t_{1}|+|t'_{2}|})}\alpha^{3}(t_{1})\otimes\alpha^{2}(t'_{2})\otimes
[\alpha(x), \alpha(t''_{2})]\textbf{)}\\
&=\alpha^{\otimes3}\textbf{(} \circlearrowleft\textbf{(}[x, \alpha(t_{1})]\otimes\Delta(\alpha(t_{2}))+(-1)^{{|x|}{|t_{1}|}}\alpha^{2}(t_{1})\otimes[x, t'_{2}]\otimes\alpha(t''_{2})
+(-1)^{{|x|}({|t_{1}|+|t'_{2}|})}\alpha^{2}(t_{1})\otimes\alpha(t'_{2})\otimes
[x, t''_{2}]\textbf{)}\textbf{)}\\
&=\alpha^{\otimes3}(\circlearrowleft ad_{x}(\alpha(t_{1})\otimes t'_{2}\otimes t''_{2}))
=\alpha^{\otimes3}(\circlearrowleft ad_{x}((\alpha\otimes\Delta)(t))).
\end{align*}
This proves (\ref{m10}).
\end{proof}
The following result is a special case of the previous theorem.
\begin{cor}\ Let $(\mathfrak{L}, [\cdot ,\cdot ], \Delta, \alpha)$ be a multiplicative Hom-Lie superbialgebra and $t\in\mathfrak{L}^{\otimes2}$ be an element such that $\alpha^{\otimes2}(t)=t, \ \ t_{21}=-t,$\  (\ref{2007}) and (\ref{777}), ((\ref{777})\ and\ (\ref{2007}),\  i.e., $|t|=\bar{0})$
and $[[t, t]]^{\alpha}+\circlearrowleft(\alpha\otimes\Delta)(t)=0$
for all $x\in\mathfrak{L}$.
Then $\mathfrak{L}_{t}=(\mathfrak{L}, [\cdot ,\cdot ], \Delta_{t}=\Delta+ad(t),  \alpha)$ is multiplicative Hom-Lie superbialgebra.
\end{cor}

\begin{small}
\textbf{Addresses}.\\ \emph{M. Fadous} , University of Sfax, Faculty of Sciences Sfax,  BP
1171, 3038 Sfax, Tunisia. \\
\emph{S. Mabrouk,} Universit\'{e} de Gafsa,  Facult\'{e} des Sciences, Gafsa Tunisia\\
\emph{A. Makhlouf,} University of Haute Alsace, 4 rue des fr\`eres Lumi\`ere, 68093 Mulhouse France. \\
\textbf{Emails.} mohamedfadous201001@gmail.com\\
mabrouksami00@yahoo.fr\\
 Abdenacer.Makhlouf@uha.fr
\end{small}

\end{document}